\documentclass{amsart}

\usepackage[utf8]{inputenc}
\usepackage{color}
\usepackage{hyperref}
\usepackage{amssymb}
\usepackage{float}
\usepackage{tikz}
\usepackage{tikz-cd}
\usepackage{forest}
\usepackage{nicefrac}
\usepackage{enumerate}
\usepackage[nameinlink, capitalize, noabbrev]{cleveref}
\usepackage{booktabs}
\usepackage{xcolor}
\usetikzlibrary{shapes.geometric,decorations.markings}
\usetikzlibrary{decorations.pathreplacing}
\usetikzlibrary{fit}
\usetikzlibrary{positioning}
\usetikzlibrary{intersections}
\tikzset{mytext/.style={font=\small, text=black}}

\hypersetup{
    colorlinks=true,
    linkcolor=teal,
    citecolor=magenta,
    }

\newtheorem{Theorem}{Theorem}

\newtheorem{Conjecture}{Conjecture}
\newtheorem{Corollary}[Theorem]{Corollary}
\newtheorem{proposition}{Proposition}[section]
\newtheorem{lemma}[proposition]{Lemma}
\newtheorem{corollary}[proposition]{Corollary}
\newtheorem{theorem}[proposition]{Theorem}
\newtheorem{Question}[Conjecture]{Question}
\newtheorem{question}[proposition]{Question}
\theoremstyle{definition}

\newtheorem{definition}[proposition]{Definition}

\newcommand{\F}{\mathbb{F}}

\DeclareMathOperator{\Aut}{Aut~}
\DeclareMathOperator{\hdim}{hdim}
\DeclareMathOperator{\hspec}{hspec}
\DeclareMathOperator{\rst}{rst}

\DeclareMathOperator{\st}{st}
\DeclareMathOperator{\St}{St}

\numberwithin{equation}{section}

\title[On finitely generated Hausdorff spectra]{
On finitely generated Hausdorff spectra of groups acting on rooted trees}
\author[J. Fariña-Asategui, O. Garaialde Ocaña and J. Uria-Albizuri]{Jorge Fariña-Asategui, Oihana Garaialde Ocaña and Jone Uria-Albizuri}
\address{Jorge Fariña-Asategui: Centre for Mathematical Sciences, Lund University, 223 62 Lund, Sweden -- Department of Mathematics, University of the Basque Country UPV/EHU, 48080 Bilbao, Spain}
\email{jorge.farina\_asategui@math.lu.se}

\address{Oihana Garaialde Ocaña: Department of Mathematics, University of the Basque Country UPV/EHU, 48080 Bilbao, Spain}
\email{oihana.garayalde@ehu.eus}

\address{Jone Uria-Albizuri: Department of Mathematics, University of the Basque Country UPV/EHU, 48080 Bilbao, Spain}
\email{jone.uria@ehu.eus}

\keywords{Hausdorff spectra, weakly branch groups, iterated wreath products, profinite groups}
\subjclass[2020]{Primary: 20E08, 20E18 ; Secondary: 28A78}
\thanks{All the authors are supported by the Spanish Government, grant PID2020-117281GB-I00, partly with FEDER funds. The first author is also supported from the Walter Gyllenberg Foundation from the Royal Physiographic Society of Lund. The second and third authors are supported partly from European
Regional Development Fund (ERDF), the Basque Government, grant
IT1483-22 and Ministry of Science, Innovation and Universities of Spain, grant  PCI2024-155096-2.}

\begin{document}

\begin{abstract}

We answer two longstanding questions of Klopsch (1999) and Shalev (2000) by proving that the (topologically)
finitely generated Hausdorff spectrum of the closure of a finitely generated regular branch group with respect to the level-stabilizer filtration is the full interval [0,1].

Previous work related to the above questions include the celebrated result of Abért and Virág on the completeness of the finitely generated Hausdorff spectrum of the group of $p$-adic automorphisms $W_p$. We extend their result to any level-transitive iterated wreath product acting on a regular rooted tree, and more importantly, provide
the first explicit examples of finitely generated subgroups with prescribed Hausdorff dimension. In fact, in $W_p$ such subgroups can be taken to be 2-generated,  giving further evidence to a conjecture of Abért and Virág.

Furthermore, we extend a well-known result of Klopsch (1999) for branch groups to closed level-transitive groups with fully dimensional rigid stabilizers: these have full Hausdorff spectrum with respect to the level-stabilizer filtration. As a consequence, we determine, for the first time, the Hausdorff spectum of many well-studied families of weakly branch groups.

As an additional new contribution, we define {\it nice filtrations} (which include the level-stabilizer filtration). In fact, the first and the last result can be proven more generally with respect to these filtrations (under possibly some further assumptions).

Finally, we also consider two standard filtrations, namely the $2$-central lower series, and the dimension subgroup or Jenning-Zassenhaus series in the closure of the first Grigorchuck group, and show that its finitely generated Hausdorff spectra with respect to these filtrations is the full interval [0,1], answering a recent question of de las Heras and Thillaisundaram (2022). This is obtained via a reduction to properties concerning to nice filtrations.

\end{abstract}

\maketitle

\section{introduction}
\label{section: introduction}

Countably based profinite groups are totally disconnected, Hausdorff and compact topological groups and so, metrizable. This, in turn, allows the definition of Hausdorff dimension of their measurable subsets. Since the pioneering work of Abercrombie in \cite{Abercrombie} and Barnea and Shalev in \cite{BarneaShalev}, Hausdorff spectra of countably based profinite groups have provided fruitful results; see \cite{Jorge, PadicAnalytic} and the references therein.

For a countably based profinite group $G$,  let $\mathcal G:=\{G_i\}_{i\geq 0}$ be a filtration series of $G$, i.e. a descending chain of open normal subgroups 
\[
G=G_0 \supseteq G_1\supseteq G_2 \supseteq \dots \supseteq \{1\}
\]
with trivial intersection. This filtration induces a translation invariant metric on~$G$, where for every $x,y\in G$, the distance between $x$ and $y$ is given by
\[
d^{\mathcal{G}}(x,y):=\inf_{i\ge 0} \big\{|G:G_i|^{-1} \mid  xy^{-1}\in G_i\big\}.
\]
Then, by \cite[Theorem 2.4]{BarneaShalev}, the {\it Hausdorff dimension of a closed subgroup $H\leq G$ with respect to $\mathcal G$} is given by the following limit inferior
\[
\hdim_G^{\mathcal G}(H)=\liminf_{i\to \infty} \frac{\log|HG_i:G_i|}{\log|G:G_i|}=\liminf_{i\to \infty} \frac{\log|H:H\cap G_i|}{\log|G:G_i|}.
\]
The {\it Hausdorff spectrum of $G$ with respect to $\mathcal G$} is the set of Hausdorff dimensions of the closed subgroups of $G$ with respect to $\mathcal{G}$, i.e.
\[
\{0,1\}\subseteq \hspec^{\mathcal G}(G)=\{\hdim_G^{\mathcal G}(H)\mid H\leq G \text{ closed}\}\subseteq [0,1].
\]

A key motivation to study the Hausdorff spectra of profinite groups stems from a conjecture of Barnea and Shalev in \cite{BarneaShalev}, which states that the class of $p$-adic analytic pro-$p$ groups can be characterized by the finiteness of their Hausdorff spectrum. Although this conjecture remains open in its full generality, it was proven in \cite{PadicAnalytic} for the case of solvable $p$-adic analytic pro-$p$ groups. Soon, the study of the Hausdorff spectra of other classes of profinite groups was started, including the Hausdorff spectra of branch profinite groups.

The class of branch groups was introduced by Grigorchuk in 1997 at the St. Andrew’s group theory conference in Bath as a common generalization of the first Grigorchuk group \cite{GrigorchukBurnside} and the Gupta-Sidki $p$-groups \cite{GuptaSidki} among others; see \cref{section: Preliminaries} for precise definitions here and elsewhere in the introduction. The class of branch groups  includes the first examples of groups with intermediate growth and amenable but not elementary amenable groups. One of the first results in  Hausdorff spectra of branch groups was obtained by Klopsch in his PhD thesis in 1999:

\begin{theorem}[{see {\cite[Theorem 1.6]{KlopschPhD}}}]
\label{theorem: Benjamin spectra}
    The Hausdorff spectrum of a profinite branch group with respect to the level-stabilizer filtration $\mathcal{S}=\{\mathrm{St}(n)\}_{n\ge 1}$ is the full interval $[0,1]$.
\end{theorem}

We noticed that the proof of \cref{theorem: Benjamin spectra} is based on a technical lemma \cite[Lemma 5.3]{KlopschPhD} that is related to rigid stabilizers, and we pin down the key conditions required in its proof. This, in turn, motivates the definition of {\it nice filtrations} (see \cref{section: full and normal hausdorff spectrum}), which, as remarkable examples in a branch group $G\le \mathrm{Aut}~T$, include the level stabilizer filtration $\mathcal{S}$ and the rigid level stabilizer filtration $\{\mathrm{Rist}_G(n)\}_{n\ge 1}$. We obtain the following result:

 \begin{Theorem}
 \label{Theorem: full spectrum}
        Let $G\le \mathrm{Aut}~T$ be a closed and level-transitive subgroup such that its rigid stabilizers are fully dimensional with respect to a nice filtration $\mathcal{N}$. Then 
        $$\mathrm{hspec}^{\mathcal{N}}(G)=[0,1].$$
\end{Theorem}

Note that this result generalizes \cref{theorem: Benjamin spectra} in both ways: it includes a wider class of groups, and it considers more general filtrations. Even more, we also determine the normal Hausdorff spectrum of some of such groups (see \cref{theorem: full spectra}), and identify main examples of branch and of weakly branch groups in the literature that satisfy the required assumptions in it.

For families of groups with a complete Hausdorff spectra, one is interested in studying the Hausdorff spectrum restricted to a subclass of subgroups. For a property $\mathcal{P}$ the $\mathcal{P}$ {\it Hausdorff spectrum of $G$ with respect to $\mathcal G$} is the set of Hausdorff dimensions of the closed subgroups of $G$ with respect to $\mathcal{G}$ satisfying the additional property $\mathcal{P}$.

Restrictions of the Hausdorff spectrum of $\mathrm{Aut}~T_p$ and its Sylow pro-$p$ subgroup~$W_p$ to self-similar subgroups, branch subgroups, etc were already considered on the early 2000s by Bartholdi, Grigorchuk and \v{S}uni\'{c} among others; compare \cite{BartholdiHausdorff, GrigorchukFinite, SunicHausdorff}. However, it was not until very recently that these restricted spectra were completely established for the group $W_p$ by the first author in \cite{Jorge}. 

Another natural restriction of the Hausdorff spectrum is to topologically finitely generated closed subgroups. Already in 1999, Klopsch asked in \cite{KlopschPhD} about the finitely generated Hausdorff spectrum of branch groups:

\begin{Question}[{see {\cite[Question 5.4]{KlopschPhD}}}]
\label{question: Benjamin}
Does there exist a topologically finitely generated branch profinite group whose finitely generated Hausdorff spectrum is uncountable?
\end{Question}

In the context of pro-$p$ groups, Shalev also asked about the finitely generated Hausdorff spectrum:

\begin{Question}[{see {\cite[Problem 17]{ShalevNewHorizons}}}]
\label{question: fgShalev}
Are there topologically finitely generated pro-p groups with uncountable finitely generated spectrum? Can the finitely generated spectrum of such a group contain irrational numbers? Intervals?
\end{Question}

Computing the finitely generated Hausdorff spectrum of a profinite group has proven to be in general a difficult problem, which can be observed in the fact that there are not many examples in the literature where the finitely generated Hausdorff spectrum of a group has been computed; compare \cite{BarneaShalev, IkerAnitha, OihanaAlejandraBenjamin}. Indeed, the only branch profinite group whose finitely generated Hausdorff spectrum is known is the aforementioned group $W_p$, thanks to the remarkable work of Abért and Virág in \cite{AbertVirag}. They showed that the 3-bounded Hausdorff spectrum of $W_p$ is the whole interval $[0,1]$. Furthermore, they showed the existence of topologically finitely generated subgroups of $W_p$ whose finitely generated Hausdorff spectrum contains any prescribed real number $\alpha\in [0,1]$, giving an answer to the second part of \cref{question: fgShalev}.

However, since $W_p$ is not topologically finitely generated (see \cite[Corollary~3.6]{Bondarenko}), their methods do not provide an answer neither to \cref{question: Benjamin} nor to the first and the last part of \cref{question: fgShalev}. Moreover, their probabilistic methods are not constructive and do not yield any explicit example of a finitely generated subgroup of $W_p$ with prescribed Hausdorff dimension. 
Nonetheless, some years later, Siegentaler constructed an explicit example for the 1-dimensional case in \cite{SiegenthalerBinary}.

Our second goal in this article is to finally answer both longstanding \textcolor{teal}{Questions}~\ref{question: Benjamin} and \ref{question: fgShalev} in a very strong sense. In fact, we prove that the closure of any finitely generated regular branch group in $\mathrm{Aut}~T_m$ has complete bounded Hausdorff spectrum:

\begin{Theorem}
\label{theorem: finitely generated spectrum of regular branch}
    Let $m\ge 2$ be a natural number and let $T_m$ be the $m$-adic tree. Let $G\le \mathrm{Aut}~T_m$ be a finitely generated regular branch group. Then there exists a natural number $b\ge 2$ such that
    $$\mathrm{hspec}_{\mathrm{b}}^\mathcal{S}(\overline{G})=\mathrm{hspec}_{\mathrm{fg}}^\mathcal{S}(\overline{G})=\mathrm{hspec}^\mathcal{S}(\overline{G})=[0,1].$$
Furthermore, if $G$ is also just infinite then for any nice filtration $\mathcal{N}$ we get
    $$\mathrm{hspec}_{\mathrm{b}}^\mathcal{N}(\overline{G})=\mathrm{hspec}_{\mathrm{fg}}^\mathcal{N}(\overline{G})=\mathrm{hspec}^\mathcal{N}(\overline{G})=[0,1].$$
\end{Theorem}

Observe that both \cite[Theorem 2]{AbertVirag} and \cref{theorem: finitely generated spectrum of regular branch} provide examples of groups for which both their bounded and their finitely generated Hausdorff spectra are complete. In view of this, we ask the following:

 \begin{Question}
    Does there exist a finitely generated profinite or pro-$p$ group whose finitely generated Hausdorff spectrum is complete but whose bounded Hausdorff spectrum is not?
\end{Question}

We point out that the proof of \cref{theorem: finitely generated spectrum of regular branch} does not make full use of the finite generation and the branchness assumptions on the group $G$. Both hypothesis are only used to ensure the existence of a finitely generated branching normal subgroup of the group $G$ with fully dimensional closure in $\overline{G}$. Here, fully dimensional means that the Hausdorff dimension in $\overline{G}$ is one. Thus, for a weakly regular branch group it suffices to find a branching subgroup satisfying the previous conditions in order to determine its bounded Hausdorff spectrum. We obtain the following corollary:

\begin{Corollary}
\label{corollary: weakly branch fg}
    Let $G\le \mathrm{Aut}~T_m$ be a weakly regular branch group branching over a finitely generated normal subgroup with fully-dimensional closure in $\overline{G}$. Then $\overline{G}$ has complete bounded Hausdorff spectrum with respect to $\mathcal{S}$.
\end{Corollary}

The well-known \textit{Basilica group} $\mathcal{B}\le \mathrm{Aut}~T_2$, i.e. the group generated by the automorphisms 
$$\psi(a)=(1,b)\quad \text{and}\quad \psi(b)=(1,a)\sigma,$$
where $\sigma=(1\,2)\in \mathrm{Sym}(2)$; satisfies the assumptions in \cref{corollary: weakly branch fg}. Indeed $\mathcal{B}$ is weakly regular branch over its commutator subgroup $\mathcal{B}'$ \cite[Theorem 1]{Gri-Zuk:Basilica}, $\mathcal{B}$ is not branch \cite[Theorem 3.5]{pBasilica}, its commutator subgroup $\mathcal{B}'$ is $3$-generated \cite[Proposition 4.7]{Dominik}, and $\overline{\mathcal{B}'}$ is fully dimensional in $\overline{\mathcal{B}}$ \cite[Theorem 5]{AbertVirag}. Thus, the Basilica group is the first example of a weakly branch group but not branch with full bounded Hausdorff spectrum.  The following natural question arises.

\begin{Question}
\label{question: weakly regular branch generalization}
    Is the finitely generated Hausdorff spectrum of finitely generated weakly regular branch profinite groups the whole interval?
\end{Question}

Note that Abért and Virág's result on the complete finitely generated Hausdorff spectrum of $W_p$ is not covered by \cref{theorem: finitely generated spectrum of regular branch}, 
as $W_p$ is not topologically finitely generated. Therefore, the next goal of this article is to extend \cref{theorem: finitely generated spectrum of regular branch} also to level-transitive iterated permutational wreath products, generalizing the result of Abért and Virág for $W_p$. Let $m\geq 2$ be a natural number and let $\mathrm{Sym}(m)$ denote the group of permutations of the set $\{1, \dots, m\}$. For a subgroup $H\le \mathrm{Sym}(m)$, we define the \textit{iterated permutational wreath product} $W_H\le \mathrm{Aut}~T_{m}$ as 
$$W_H:=\{g\in \mathrm{Aut}~T_{m}\mid g|_v^1\in H \text{ for every }v\}.$$ 
Note that $W_H\cong\varprojlim H\wr \overset{n}{\dotsb}\wr H$. We generalize \cite[Theorem 2]{AbertVirag} to the group $W_H$ for an arbitrary transitive subgroup $H\le \mathrm{Sym}(m)$:

\begin{Theorem}
\label{Theorem: WH main result}
    Let $H\le \mathrm{Sym}(m)$ be a transitive group  with $m\ge 2$. Then there exists $b\ge 2$ such that
    $$\mathrm{hspec}_{\mathrm{b}}^\mathcal{S}(W_H)=\mathrm{hspec}_{\mathrm{fg}}^\mathcal{S}(W_H)=\mathrm{hspec}^\mathcal{S}(W_H)=[0,1].$$
\end{Theorem} 

In contrast to Abert and Virag's methods in \cite{AbertVirag}, the novelty of our approach is that we explicitly construct finitely generated subgroups of $W_H$ with a prescribed Hausdorff dimension. To that aim, we first extend the construction of the family of \v{S}uni\'{c} groups in $W_p$ (see \cite{SunicHausdorff}) to a family of groups in $W_H$. Then, we employ Siegenthaler's arguments in \cite{SiegenthalerBinary} in this more general context to obtain fully dimensional subgroups of $W_H$. The last steps of the proof rely on the same construction as in the proof of \cref{theorem: finitely generated spectrum of regular branch}.

In particular, \cref{Theorem: WH main result} yields the first explicit examples of finitely generated groups with closures of any prescribed Hausdorff dimension in $W_p$. Furthermore, for $W_p$ we are able to refine our construction to get $b=2$, giving further evidence to a conjecture of Abért and Virág \cite[Conjecture 7.3]{AbertVirag}:

\begin{Corollary}
\label{Corollary: Wp 2-bounded spectrum}
    Let $p$ be any prime number. Then 
    $$\mathrm{hspec}_{2}^\mathcal{S}(W_p)=[0,1].$$
\end{Corollary}

To conclude the article, observe that the closure of many finitely generated branch groups are pro-$p$ groups themselves, and as such, they admit other natural  filtrations. We consider the closure of the first Grigorchuck group and two of the main standard filtrations on pro-$p$ groups: the $p$-central lower series $\mathcal{L}$ and the dimension subgroup or Jenning-Zassenhaus series $\mathcal{D}$. Let $\mathfrak{G}$ denote the first Grigorchuk group which acts on the binary rooted tree $T_2$ and is generated by four elements; say, $a,b,c,d$, where $a$ is the rooted automorphism given by $a:=(1\,2)\in \mathrm{Sym}(2)$, i.e. $a$ permutes the two subtrees hanging below the root, and the rest of the elements can recursively be defined as 
\[
\psi(b)=(a,c), \quad \psi(c)=(a,d)\quad \text{and}\quad  \psi(d)=(1,b).
\]
It is known that $\mathfrak{G}$ is  regular branch over its normal subgroup $K=\langle abab\rangle^\mathfrak{G}$; see \cite[Proposition 4.2]{BartholdiGrigorchuk2}. The lower 2-central series $\mathcal{L}$ and the Jenning-Zassenhaus series $\mathcal{D}$ of $\mathfrak{G}$ have been computed in \cite{BarthGrig} and in \cite{Rozhkov}. We obtain the following result:

\begin{Theorem}
\label{theorem: Grigorchuk spectra}
    Let $\Gamma\le \mathrm{Aut}~T_2$ be the closure of the first Grigorchuk group $\mathfrak{G}$ and let $\mathcal{L}$, $\mathcal{D}$ and $\mathcal S$ be its 2-central lower, Jenning-Zassenhaus and level stabilizer filtrations, respectively. Then there exists $b\ge 2$ such that
    $$\mathrm{hspec}_\mathrm{b}^\mathcal{L}(\Gamma)=\mathrm{hspec}_\mathrm{b}^\mathcal{D}(\Gamma)=\mathrm{hspec}_\mathrm{b}^\mathcal{S}(\Gamma)=[0,1].$$
\end{Theorem}

\cref{theorem: Grigorchuk spectra} answers \cite[Problem 1.2]{IkerAnitha} for the $p$-central lower series and the Jenning-Zassenhaus series. Finally, looking at both \cref{theorem: finitely generated spectrum of regular branch} and \cref{theorem: Grigorchuk spectra}, we ask the following natural questions:

\begin{Question}
\label{question: other filtrations}
Is the Hausdorff spectrum of $\Gamma$ the whole interval $[0,1]$ with respect to other filtrations such as the Frattini or the $p$-power filtrations? Does, more generally, this last result hold for topologically finitely generated regular branch profinite groups?
\end{Question}

One of the major obstacles to answer \cref{question: other filtrations} is that the standard filtrations of branch groups are difficult to compute, and only a few of them have been computed; compare \cite{BartholdiLower, BarthGrig,  GustavoMikelMarialaura, Rozhkov,Vieira}.

\subsection*{\textit{\textmd{Organization}}} In \cref{section: Preliminaries} we set a background for the subsequent sections. We introduce the concept of nice filtrations in \cref{section: full and normal hausdorff spectrum} and prove \cref{Theorem: full spectrum}. \cref{section: finitely generated Hausdorff spectrum} is devoted to the study of the bounded Hausdorff spectrum of regular branch groups, where we prove \cref{theorem: finitely generated spectrum of regular branch}. In \cref{section: iterated wreath products} we deal with iterated wreath products and prove \cref{Theorem: WH main result}. Finally, in \cref{subsection: grigorchuk group}, we work with some of the standard filtrations in pro-$p$ groups and study the Hausdorff spectra of the first Grigorchuk group with respect to these standard filtrations, proving \cref{theorem: Grigorchuk spectra}.

\subsection*{\textit{\textmd{Notation}}} Throughout the text $p$ denotes a prime number. The integer $m\ge 2$ is used to denote the degree of the regular tree $T_m$, except in \cref{subsection: grigorchuk group} where it is used as a subindex for a filtration of the first Grigorchuk group acting on the binary rooted tree. All the logarithms will be taken in base $m$ when working with subgroups of $\mathrm{Aut}~T_m$,  and in base 2 otherwise, as in the non-regular case every result is independent of the base of the logarithms.  Maps will be considered to act on the right and we write composition from left to right, i.e. $fg$ means first apply $f$ and then apply $g$. Therefore, conjugation of $g$ by $h$ is given by $g^h=h^{-1}gh$ and commutators by $[g,h]=g^{-1}h^{-1}gh$. We shall use the notation $H\trianglelefteq G$ to denote a normal subgroup. For a finite set $S$ we denote its cardinality by $\#S$. Lastly, the topology of any group acting on the $m$-adic tree is always assumed to be the \textit{congruence topology}, i.e. the one given by its level stabilizers.

\subsection*{\textit{\textmd{Aknowledgements}}}
The second author thanks Alejandra Garrido for revealing Rozhkov's paper \cite{Rozhkov} while working in the project \cite{OihanaAlejandraBenjamin}.

\section{Preliminaries}
\label{section: Preliminaries}

In this section we present the main objects under study, namely groups acting on rooted trees; see \cite{Handbook,SelfSimilar} for an in depth introduction to these groups. As the automorphism group of a rooted tree is a countably based profinite group, we additionally introduce the notion of Hausdorff dimension in this context (see \cite{Jorge, PadicAnalytic} and the references therein) and we state some useful results that will be employed throughout the manuscript.

\subsection{The $\mathbf{m}$-adic tree and its group of automorphisms} For a sequence of natural numbers $\mathbf{m}=\{m_i\}_{i\ge 0}$ we define the $\mathbf{m}$-\textit{adic tree} $T_{\mathbf{m}}$ as the infinite rooted tree where each vertex at distance $i$ from the root has exactly $m_i$ descendants. If $m_i=m$ is constant we call it the \textit{$m$-adic tree} and denote it by $T_m$. For every $i\geq 1$, let $X_i$ be a set of $m_i$ elements. Then, each vertex of the {\bf m}-adic tree can be identified with a finite word on  the sets $X_i$. Moreover, the sets $X_i$ may be completely ordered, inducing a graded lexicographical order in the $\textbf{m}$-adic tree. Two vertices $u$ and~$v$ are joined by an edge if $u=vx$ with $x\in X_{d(v)}$. More generally, we say $u$ is below $v$ if $u=vx_1\dotsb x_r$ for some $x_i\in X_{d(v)+i-1}$. The \textit{$n$th level of the tree} $\mathcal{L}_n$ is defined to be the set of all vertices which consists of words of length $n$. We may also use the word level to refer to the number $n$. Note that the number of vertices in the $n$th level is precisely $N_n:=\prod_{i=0}^{n-1}m_i$. We shall also write $d(v)$ for the level at which a vertex $v\in T_{\mathbf{m}}$ lies.  For any integer $1\le k\le \infty$ and any vertex $v$, the \textit{$k$th truncated tree at $v$}, denoted by $T_{\mathbf{m}}^{[k,v]}$, and simply by $T_{\mathbf{m}}^{[v]}$ for $k=\infty$, consists of the vertices below $v$ at distance at most $k$ from $v$. Similarly, we will denote the tree $T_{\mathbf{m}}^{[k,\emptyset]}$ consisting of the first $k$ levels just by  $T_{\mathbf{m}}^{[k]}$, where here  $\emptyset$ denotes the root of the tree.

For a rooted tree $T$, let $\mathrm{Aut}~T$ denote the group of graph automorphisms of $T$, i.e. bijective maps on the set of vertices of $T$ preserving adjacency and fixing the root. Such automorphisms permute the vertices at the same level of the rooted tree.

Let $g\in\mathrm{Aut}~T_{\mathbf{m}}$. For any vertex $v$ and any integer $1\le k\le \infty$, we define the \textit{section of $g$ at $v$ of depth }$k$ as the unique automorphism $g|_v^k\in\mathrm{Aut}~T_{\mathbf{m}}^{[k,v]}$ such that $(vu)g=(v)g(u)g|_v^k$, for every $u$ of length at most $k$. For $k=\infty$, we write  $g|_v$ for simplicity, and we say that it is the \textit{section of $g$ at $v$}.  For $k=1$, we we say that $g|_v^1$ is the \textit{label of $g$ at $v$}. Note that with this notation, for every $g,h\in\Aut T_{\mathbf{m}}$, every vertex $v\in T_{\mathbf{m}}$ and $1\le k\le\infty$, the following relations hold:

\begin{align*}\label{eq:sections}
    (gh)|_v^k&=g|_v^k \cdot h|_{(v)g}^k,\\
    (g^{-1})|_v^k&=(g|_{(v)g^{-1}}^k)^{-1},\\
    (g^h)|_v^k&=(h|_{(v)h^{-1}}^k)^{-1}\cdot g|_{(v)h^{-1}}^k\cdot h|_{(v)h^{-1}g}^k.
\end{align*}

For any vertex $v$ of the tree $T_{\mathbf{m}}$, we define its \textit{vertex stabilizer} $\mathrm{st}(v)$ as the subgroup of automorphisms fixing the vertex $v$.
Similarly, for every $n\ge 1$, we define $\mathrm{St}(n)$ to be the set of automorphisms fixing all the vertices at the $n$th level of the tree, and we say that it is the \textit{$n$th level stabilizer} of $\mathrm{Aut}~T_{\mathbf{m}}$. In fact, 
for every $n\geq 1$, the subgroup $\mathrm{St}(n)$ is normal and of finite index in $\mathrm{Aut}~T_{\mathbf{m}}$, and we have that $\mathrm{St}(n)=\bigcap_{v \in \mathcal{L}_n } \mathrm{st}(v)$.

Since $\bigcap_{n\geq 1} \mathrm{St}(n)=1$, the descending chain $\{\mathrm{St}(n)\}_{n\geq 1}$ of normal subgroups of $\mathrm{Aut}~T_{\mathbf{m}}$ forms a filtration and so,  the group $\mathrm{Aut}~T_{\mathbf{m}}$ is residually finite. Furthermore $\mathrm{Aut}~T_\mathbf{m}$ is a countably-based profinite group with respect to the topology induced by the above filtration. 

For each vertex $v$ of $T_{\mathbf{m}}$ we shall define the continuous group homomorphism 
\[\begin{matrix} \varphi_v\colon &\mathrm{st}(v)&\to & \mathrm{Aut}~T_{\mathbf{m}}^{[v]}\\
&g&\mapsto &g|_v.
\end{matrix}
\]
For each integer $n\ge 1$, let $v_1,\dotsc,v_{N_n}$ denote the vertices at the $n$th level of the tree from left to right according to the graded lexicographical order. We then define the continuous isomorphisms
 \begin{equation*}
\begin{matrix} \psi_n\colon &\mathrm{St}(n)&\to & \mathrm{Aut}~T_{\mathbf{m}}^{[v_1]}\times\overset{N_n}{\ldots}\times \mathrm{Aut}~T_{\mathbf{m}}^{[v_{N_n}]}\\
&g&\mapsto &(g|_{v_1},\dotsc,g|_{v_{N_n}}).
\end{matrix}
\end{equation*}
 In what follows, we shall write $\psi:=\psi_1$ for simplicity.

An automorphism $g\in \mathrm{Aut}~T_\mathbf{m}$ is called \textit{finitary of depth $n$} if $g|_{v}^1=1$ for every $v$ at level $k\ge n$. Finitary automorphisms of depth 1 are called \textit{rooted automorphisms} and they may be identified with elements in the symmetric group $\mathrm{Sym}(m_0)$. More generally, the subgroup of finitary automorphisms of depth $n$ is isomorphic to the group $\mathrm{Aut}~T_{\mathbf{m}}^{[n]}$ via the restriction of the action of $\mathrm{Aut}~T_{\mathbf{m}}$ to the $n$th truncated tree at the root. This allows us to extend the isomorphisms 
$$\psi_n:\mathrm{St}(n)\to \mathrm{Aut}~T_{\mathbf{m}}^{[v_1]}\times\overset{N_n}{\ldots}\times \mathrm{Aut}~T_{\mathbf{m}}^{[v_{N_n}]}$$
to the isomorphisms 
$$\psi_n:\mathrm{Aut}~T_{\mathbf{m}}\to (\mathrm{Aut}~T_{\mathbf{m}}^{[v_1]}\times\overset{m^n}{\ldots}\times \mathrm{Aut}~T_{\mathbf{m}}^{[v_{N_n}]})\rtimes \mathrm{Aut}~T_{\mathbf{m}}^{[n]},$$
which,  by a slight abuse of notation, we shall also denote by $\psi_n$.

\subsection{Subgroups of $\mathrm{Aut}~T_{\mathbf{m}}$}\label{subsec: subgroups}
Let $v$ be a vertex of $T_{\mathbf{m}}$ and let $n\geq 1$ be an integer. For a subgroup $G\le \mathrm{Aut}~T_{\mathbf{m}}$, we define 
$$\mathrm{st}_G(v):=\mathrm{st}(v)\cap G \quad \text{and}\quad \mathrm{St}_G(n):=\mathrm{St}(n)\cap G.$$
The quotients $G/\mathrm{St}_G(n)$ are called the \textit{congruence quotients} of $G$.

Similarly, for any vertex $v$, let $\mathrm{rist}_G(v)$ be the associated \textit{rigid vertex stabilizer}, i.e. the subgroup of $\mathrm{st}_G(v)$ consisting of automorphisms $g\in \mathrm{st}_G(v)$ such that $g$ only acts non-trivially on the subtree rooted at $v$. Then, for any $n\ge 1$, we define the corresponding \textit{rigid level stabilizers} $\mathrm{Rist}_G(n)$ as the product of all rigid vertex stabilizers of the vertices at the level~$n$. Note that rigid level stabilizers are normal subgroups of $G$ as any conjugate of a rigid vertex stabilizer is a rigid vertex stabilizer of the same level. If $G$ is a closed subgroup of $\mathrm{Aut}~T_{\mathbf{m}}$ then both the rigid vertex and the rigid level stabilizers are closed subgroups of $G$.

 \begin{definition}
Let $G\leq \mathrm{Aut}~T_{\mathbf{m}}$ be a subgroup. We say that $G$ is
\begin{enumerate}[\normalfont(i)]
\item  \textit{level-transitive} if $G$ acts transitively on every level of the tree;
\item \textit{weakly branch} if $G$ is level-transitive and its rigid level stabilizers are non-trivial for every level; 
\item  \textit{branch} if it is weakly branch and its rigid stabilizers are of finite index in $G$. 
\end{enumerate}
 \end{definition}

Now let us consider the $m$-adic tree $T_{m}$ and let $G$ be a subgroup of $\mathrm{Aut}~T_{m}$. We say that $K\leq G$ is a {\it branching subgroup of $G$} if $\psi(\mathrm{St}_K(1))\ge K\times\dotsb\times K$.

\begin{definition}
We say that $G\le \mathrm{Aut}~T_m$ is 
\begin{enumerate}[\normalfont(i)]
    \item \textit{self-similar} if for any $g\in G$ and any vertex $v$, we have $g|_{v}\in G$;
    \item \textit{fractal} if $G$ is self-similar and for every vertex $v\in T_m$ and every $g\in G$, there exists $h\in \mathrm{st}_G(v)$ such that $h|_v=g$;
    \item \textit{strongly fractal} if $G$ is self-similar and for every vertex $v\in \mathcal{L}_1$ and every $g\in G$, there exists $h\in \mathrm{St}_G(1)$ such that $h|_v=g$; 
    \item \textit{(weakly) regular branch} if $G$ is self-similar, level-transitive group containing a (non-trivial) finite-index branching subgroup $K$.
\end{enumerate}
\end{definition}

A \textit{just infinite discrete} group is an infinite discrete group such that all its proper quotients are finite. We define \textit{just infinite profinite} groups as those infinite profinite groups whose normal non-trivial closed subgroups are open, i.e. infinite profinite groups whose continuous proper quotients are all finite.

\subsection{Hausdorff dimension}

Countably-based profinite groups are metrizable and this allows the definition of Hausdorff measures and therefore, a Hausdorff dimension for its subsets. It was proved by Barnea and Shalev in \cite{BarneaShalev}, based on the work of Abercrombie \cite{Abercrombie}, that for a profinite group the Hausdorff dimension of its closed subgroups coincides with their lower box dimension.

For a countably-based profinite group $G$ and a filtration of open normal subgroups $\{G_n\}_{n\ge 1}$ of $G$, the formula for the Hausdorff dimension of a closed subgroup $H\le G$ with respect to $\{G_n\}_{n\ge 1}$ reads as the following limit inferior
\begin{align}
\label{align: hdim def}
    \mathrm{hdim}^{\{G_n\}_{n\ge 1}}_{G}(H)=\liminf_{n\to \infty}\frac{\log|HG_n:G_n|}{\log|G:G_n|}.
\end{align}
If the above limit exists, then we say that $H$ has \textit{strong} Hausdorff dimension in $G$. Note that if the Hausdorff dimension of a subgroup $H$ is $1$, then it is strong.

Let $\mathcal{P}$ be a property that may be satisfied by the closed subgroups of $G$. We define the \textit{$\mathcal{P}$ Hausdorff spectrum} of $G$ (with respect to $\{G_n\}_{n\ge 1}$) as the set of the values of the Hausdorff dimensions of its closed subgroups (with respect to $\{G_n\}_{n\ge 1}$ that satisfy property $\mathcal{P}$). We are interested in the normal, the finitely generated and the $b$-bounded Hausdorff spectrum,  which we shall denote them by
$$\mathrm{hspec}_{\trianglelefteq}^{\{G_n\}_{n\ge 1}}(G),\quad \mathrm{hspec}_{\mathrm{fg}}^{\{G_n\}_{n\ge 1}}(G)\quad \text{and}\quad \mathrm{hspec}_{\mathrm{b}}^{\{G_n\}_{n\ge 1}}(G),$$
respectively. Here, for the $b$-bounded Hausdorff spectrum we only consider closed subgroups topologically generated by at most $b$ elements for some $b\ge 1$. 

As we mentioned before, we are interested in studying Hausdorff spectra of (topological closures of) certain subgroups of the automorphism group of a rooted tree. So, we will state a result that provides an explicit formula to compute the Hausdorff dimension of such subgroups (see Theorem \ref{thm: Jorge_strong} below).  Let $m\geq 2$ be an integer. For a subgroup $H\le \mathrm{Sym}(m)$, we define the \textit{iterated permutational wreath product} $W_H\le \mathrm{Aut}~T_{m}$ as 
$$W_H:=\{g\in \mathrm{Aut}~T_{m}\mid g|_v^1\in H \text{ for every }v\in T_m\}.$$ 
In particular, $W_{\mathrm{Sym}(m)}$ is simply $\mathrm{Aut}~T_{m}$. Note also that $W_H\cong\varprojlim H\wr \overset{n}{\dotsb}\wr H$.  For a prime number $p$, let $\sigma:=(1\,2\,\dotsb \, p)\in \mathrm{Sym}(p)$. Then we define the group of \textit{$p$-adic automorphisms} as $W_p:=W_{\langle \sigma\rangle}$.

We specialize to the case where $G\le W_H\le\mathrm{Aut}~T_m$. Consider the level stabilizers filtration $\mathcal{S}:=\{\mathrm{St}_G(n)\}_{n\ge 1}$. As $\mathrm{St}(n)\cap G=\mathrm{St}_G(n)$ we will abuse the notation and call $\mathcal{S}$ both to the filtration of $\Aut T_{m}$ and to the induced one in $G$.

Now, define, for each $n\geq 1$, the sequence of real numbers
$$\mathfrak{s}_n:=m\log|\mathrm{St}_G(n-1):\mathrm{St}_G(n)|-\log|\mathrm{St}_G(n):\mathrm{St}_G(n+1)|$$
first introduced in \cite{GeneralizedBasilica}. We consider the ordinary generating function 
\[
S_{G}(x):=\sum_{n\ge 1}\mathfrak{s}_nx^n.
\]
The following result was proven by the first author in \cite[Theorem B]{Jorge}. 

\begin{theorem}
\label{thm: Jorge_strong}
Assume that $G\le W_H\le\mathrm{Aut}~T_m$ is either self-similar or branching, then the Hausdorff dimension of the closure of $G$ in $W_H$ with respect to $\mathcal{S}$ is strong and it may be computed simply as

$$\mathrm{hdim}_{W_H}^\mathcal{S}(\overline{G})=\frac{1}{\log|H|}\left(\log|G:\mathrm{St}_G(1)|-S_{G}(1/m)\right).$$
\end{theorem}

The sequence $\{\mathfrak{s}_n\}_{n\ge 1}$ can also be described as the forward gradient of the sequence $\{\mathfrak{r}_n\}_{n\ge 1}$ given by
$$\mathfrak{r}_n:=m\log|G:\mathrm{St}_G(n-1)|-\log|G:\mathrm{St}_G(n)|+\log|G:\mathrm{St}_G(1)|.$$
For self-similar groups, the sequence $\{\mathfrak{r}_n\}_{n\ge 1}$ is non-negative and increasing, and for branching groups, it is non-positive and decreasing; see \cite[Proposition 1.1]{Jorge}. In fact, thanks to \cite[Equation 3.1]{Jorge}, for a self-similar group $G\le \mathrm{Aut}~T_m$, we have
$$\mathfrak{r}_n=\log|G\times\overset{m}{\dotsb}\times G:\mathfrak{S}_n|\le \log|G\times\overset{m}{\dotsb}\times G:\psi(\mathrm{St}_G(1))|,$$
for every $n\ge 1$, where
$$\mathfrak{S}_n:=\psi(\mathrm{St}_G(1))\big(\mathrm{St}_G(n-1)\times\overset{m}{\dotsb}\times\mathrm{St}_G(n-1)\big).$$

\begin{lemma}[{{\cite[Lemma 3.5]{PadicAnalytic}}}]
    \label{lemma: strong dimension then product}
    Let $G$ be a countably-based profinite group with closed subgroups $K\le H\le G$. Consider a filtration series  $\{G_n\}_{n\ge 1}$ and the induced one $\{H_n\}_{n\geq 1}$ in $H$, where $H_n=H\cap G_n$. If either $K$ has strong Hausdorff dimension in $H$ or $H$ has strong Hausdorff dimension in $G$, then
    $$\mathrm{hdim}_G^{\{G_n\}_{n\ge 1}}(K)=\mathrm{hdim}_G^{\{G_n\}_{n\ge 1}}(H)\cdot \mathrm{hdim}_H^{\{H_n\}_{n\ge 1}}(K).$$
\end{lemma}
\begin{proof}
    It follows from the equality
    $$\frac{\log |K:K\cap G_n|}{\log| G:G_n|}=\frac{\log |H:H_n|}{\log |G:G_n|}\cdot \frac{\log |K:K\cap G_n|}{\log |H: H_n|}$$
    by taking lower limits, as one of the factors on the right hand side converges by assumption.
\end{proof}

The following lemma shows that taking the topological closure of a finite index subgroup does not decrease the index:

\begin{lemma}
    \label{lemma: indices grow in closures}
    Let $G\le \mathrm{Aut}~T_{\mathbf{m}}$ and $H\le G$ a subgroup of finite index. Then $ |\overline{G}:\overline{H}|\le |G:H|$ where the closure is taken with respect to the congruence topology.
\end{lemma}
\begin{proof}
    Let us write $G=\bigsqcup_{i=1}^ng_iH$. Then
    $$\overline{G}\supseteq \bigcup_{i=1}^ng_i\overline{H}\supseteq \bigsqcup_{i=1}^ng_iH=G$$
    and $\bigcup_{i=1}^ng_i\overline{H}$ is closed being the finite union of closed sets. Then since $\bigcup_{i=1}^ng_i\overline{H}$ contains $G$ and it is closed, it must contain the closure of $G$ and $\overline{G}= \bigcup_{i=1}^ng_i\overline{H}$ so the result follows.
\end{proof}

The following lemma is a well-known result; see for instance \cite[Lemma 4]{pro-c} or \cite[Lemmata 2.16 and 2.17]{Dominik} for a proof.

\begin{lemma}
    \label{lemma: normal subgroups contain derived of rist}
    Let $G\le \mathrm{Aut}~T_{\mathbf{m}}$ be a level-transitive subgroup and let $N\trianglelefteq G$ be a normal subgroup. Then there exists $n\ge 1$ such that $N\ge \mathrm{Rist}_G(n)'$. Furthermore if $G$ is weakly branch then $\mathrm{Rist}_G(n)^{(k)}\ne 1$ for every $n,k\ge 1$.
\end{lemma}

\section{Nice filtrations and Hausdorff dimension}
\label{section: full and normal hausdorff spectrum}

In this section we develop the tools we shall use in the rest of the paper. Our approach generalizes the one of Klopsch in \cite[Chapter VIII, Section 5]{KlopschPhD} and motivates the definition of nice filtrations as those filtrations for which the rigid stabilizers behave as nicely as for the level-stabilizer filtration. This turns out to be fruitful on the study of the finitely generated Hausdorff spectrum of regular branch profinite groups. We conclude the section with an application of the tools developed in this section which generalizes both a result of Klopsch \cite[Chapter VIII Theorem~1.6]{KlopschPhD} and a result of Abért and Virág \cite[Theorem 7]{AbertVirag}.

\subsection{Definition and main examples of nice filtrations}\label{sec: NiceDefEx}

Let $G\le \mathrm{Aut}~T_{\mathbf{m}}$ be a closed subgroup and let $\{G_i\}_{i\ge 0}$ be a filtration of open normal subgroups of $G$.

\begin{definition}[Nice filtration]
\label{definition: nice filtration}
We say that $\{G_i\}_{i\ge 0}$ is a {\it nice filtration} of $G$ if, for every $k\geq 1$, there is an isomorphism $$\frac{\mathrm{Rist}_G(k)G_i}{G_i}\cong \frac{\mathrm{rist}_G(v_1)G_i}{G_i}\times\dotsb \times \frac{\mathrm{rist}_G(v_{N_k})G_i}{G_i}$$
for almost all $i\ge 1$. We denote a nice filtration $\{G_i\}_{i\geq 1}$ of $G$ by $\mathcal{N}$.
\end{definition}

We say that two vertices $u$ and $v$ are \textit{comparable} if they are equal or one is below the other. A subset $V\subset T_{\mathbf{m}}$ is said to be \textit{non-comparable} if every pair of distinct vertices $u,v\in V$ is non-comparable.

An immediate consequence of \cref{definition: nice filtration} is that, for a nice filtration $\mathcal{N}$, any $k\ge 1$ and any non-comparable $V\subseteq \mathcal{L}_k$, the equality
\begin{align}
    \label{align: nice filtration}
\log|\prod_{v\in V}\mathrm{rist}_G(v)G_i:G_i|=\sum_{v\in V}\log|\mathrm{rist}_G(v)G_{i}:G_{i}|
\end{align} 
holds for almost all $i\ge 1$. We consider two main sources for nice filtrations of~$G$:
\begin{enumerate}[{\normalfont(i)}]
    \item The level stabilizers $\mathcal{S}$ and 
    \item \textit{direct product filtrations}, i.e. filtrations $\{G_i\}_{i\ge 0}$ such that for every $k\ge 1$, the equality
$$G_i=(G_i\cap \mathrm{rist}_G(v_1))\times\dotsb\times (G_i\cap \mathrm{rist}_G(v_{N_k}))$$
holds, for almost all $i\ge 0$.
\end{enumerate}

On the one hand, the filtration $\mathcal{S}$ is the most natural filtration for groups acting on rooted trees and thus, the main filtration considered in the literature; compare \cite{AbertVirag, Jorge, GGS, KlopschPhD, SunicHausdorff}. On the other hand, for branch profinite groups the main example of a direct product filtration is the one given by the rigid level stabilizers $\{\mathrm{Rist}_G(n)\}_{n\ge 0}$. In the case of regular branch profinite groups one may construct direct product filtrations in a very natural way via finite index normal branching subgroups. Let $K\le G$ be a finite index normal branching subgroup. Then, setting $G_0:=G$, $G_1:=K$ and for $i\ge 2$,
$$G_i:=\psi^{-1}(G_{i-1}\times \dotsb \times G_{i-1}),$$
 we obtain a natural direct product filtration.

There exist examples of natural filtrations which are not nice, specially in the case of self-similar pro-$p$ groups. For instance, for the closure of the first Grigorchuk group, its 2-central lower series $\mathcal{L}$ and its Jenning-Zassenhaus series $\mathcal{D}$ are not nice filtrations; see \cref{subsection: grigorchuk group}.

\subsection{General results for nice filtrations}
Let $G\leq \mathrm{Aut}~T_{\mathbf{m}}$ be a closed subgroup. We assume in the remainder of the paper that the filtration $\mathcal{N}$ is an arbitrary nice filtration of $G$ unless otherwise stated.  Recall that for a vertex $v\in T_{\mathbf{m}}$ we write $N_{d(v)}$ for the number of vertices of $T_{\mathbf{m}}$ at the level at which $v$ lies.

\begin{lemma}
\label{lemma: dimension of product}
    Let $G\le \mathrm{Aut}~T_{\mathbf{m}}$ be a level-transitive closed subgroup such that all its rigid level stabilizers are 1-dimensional in $G$ with respect to the filtration $\mathcal{N}$. Then, for every vertex $v\in T_{\mathbf{m}}$, the rigid vertex stabilizers have strong Hausdorff dimension in $G$ given by
    $$\mathrm{hdim}^\mathcal{N}_G(\mathrm{rist}_G(v))=\frac{1}{N_{d(v)}}.$$
\end{lemma}
\begin{proof}
Since $G$ is a level-transitive subgroup of $\mathrm{Aut}~T_{\mathbf{m}}$, the rigid vertex stabilizers of the same level are all conjugate and thus
$$\log|\mathrm{rist}_G(v_j)G_i:G_i|=\log|\mathrm{rist}_G(v_k)G_i:G_i|$$
for every $v_j,v_k$ at the same level of $T_{\mathbf{m}}$. From \cref{align: nice filtration}, it follows that
\begin{align}
\label{align: equality rists}
    \log|\mathrm{Rist}_G(k)G_i:G_i|&=\sum_{j=1}^{N_k}\log|\mathrm{rist}_G(v_j)G_{i}:G_{i}|\\
    &=N_k\cdot\log|\mathrm{rist}_G(v_1)G_{i}:G_{i}|\nonumber,
\end{align}
for almost all $i\ge 0$. Then since 1-dimensional subgroups have strong Hausdorff dimension, the result follows from \cref{align: equality rists}  by taking the limit on \cref{align: hdim def}.
\end{proof}

In what follows, let $V\subseteq T_{\mathbf{m}}$ be non-comparable  and let $\omega:=\{\omega_v\}_{v\in V}$ be a set of weights such that $0\le \omega_v\le 1$ for every $v\in V$. We define the \textit{$\omega$-measure} of $V$ as
$$\mu_{\omega}(V):=\sum_{v\in V}\frac{\omega_v}{N_{d(v)}}\in [0,1].$$
The quantity $\mu_\omega(V)$ is well-defined. Indeed, if we denote by $V_n$ the set of vertices of~$V$ at distance at most $n$ from the root, then the sequence $\{\mu_\omega(V_n)\}_{n\ge 0}$ is bounded and non-decreasing. Hence, by monotone convergence, we have that
$$\mu_{\omega}(V)=\lim_{n\to\infty} \mu_{\omega}(V_n).$$
If,  for all $v\in V$, we have $\omega_v=1$,  we just write $\mu(V)$. The next result follows from the definition of $\mu$:

\begin{lemma}
\label{lemma: existence of V}
    For any $\gamma\in [0,1]$, there exists a non-comparable subset $V\subseteq T_\mathbf{m}$ such that $\mu(V)=\gamma$.
\end{lemma}

Let $\{H_v\}_{v\in V}$ be a family of closed subgroups of $G$  such that $H_v\le \mathrm{rist}_G(v)$ for every $v\in V$. We define the closed subgroup $H_V\le G$ via
$$H_V:=\overline{\langle H_v\mid v\in V\rangle}=\prod_{v\in V}H_v,$$
where the last equality holds as each $H_v$ is closed and $V$ is non-comparable.

\begin{proposition}
\label{proposition: dimension is a series}
    Let $G\le \mathrm{Aut}~T_{\mathbf{m}}$ be a group satisfying the assumptions in \cref{lemma: dimension of product} and let $V\subseteq T_{\mathbf{m}}$ be a non-comparable subset. Assume further that for every $v\in V$ there exist $H_v\le \mathrm{rist}_G(v)$ and $0\le \omega_v\le 1$ such that $H_v$ has strong Hausdorff dimension in $G$ given by
    $$\mathrm{hdim}_{G}^\mathcal{N}(H_v)=\omega_v\cdot \mathrm{hdim}_{G}^\mathcal{N}(\mathrm{rist}_G(v)).$$
    Then
    $$\mathrm{hdim}^\mathcal{N}_G(H_V)=\mu_{\omega}(V).$$
\end{proposition}
\begin{proof}
First we prove the statement for $V$ finite. It is enough to prove the result for $\prod_{v\in V}\mathrm{rist}_G(v)$ as the general case for $H_V$ follows from this one by the strongness of the Hausdorff dimension of $H_v$ in $G$. For every $v\in T_{\mathbf{m}}$ and every $n\ge d(v)$, we write $U_{v,n}$ for the set of descendants of $v$ at level $n$. We prove by induction on the cardinality of $V$ that $\prod_{v\in V}\mathrm{rist}_G(v)$ and $\prod_{v\in V}\prod_{u\in U_{v,n}}\mathrm{rist}_G(u)$ have the same Hausdorff dimension in $G$ for any $n\ge \max_{w\in V}\{d(w)\}$. Then the result follows from applying \cref{align: nice filtration} and \cref{lemma: dimension of product} to the subgroup $\prod_{v\in V}\prod_{u\in U_{v,n}}\mathrm{rist}_G(u)$.

Let us prove first the base case $V=\{v\}$. From \cref{align: nice filtration} and \cref{lemma: dimension of product} we get
\begin{align*}
    \mathrm{hdim}_{G}^\mathcal{N}(\mathrm{rist}_G(v))&=\lim_{i\to \infty}\frac{\log|\mathrm{rist}_G(v)G_i:G_i|}{\log|G:G_i|}=\frac{1}{N_{d(v)}}=\sum_{u\in U_{v,n}}\frac{1}{N_{d(u)}}\\
    &=\sum_{u\in U_{v,n}}\lim_{i\to \infty}\frac{\log|\mathrm{rist}_G(u)G_i:G_i|}{\log|G:G_i|}
    \\
    &=\lim_{i\to \infty}\frac{\log|\big(\prod_{u\in U_{v,n}}\mathrm{rist}_G(u)\big)G_i:G_i|}{\log|G:G_i|}\\
    &=\mathrm{hdim}_{G}^\mathcal{N}(\prod_{u\in U_{v,n}}\mathrm{rist}_G(u)),
\end{align*}
i.e. $\mathrm{rist}_G(v)$ and $\prod_{u\in U_{v,n}}\rst_G(u)$ have the same Hausdorff dimension in $G$; proving the base case. Let us assume by induction that $\prod_{w\in V} \mathrm{rist}_G(w)$ has the same Hausdorff dimension as $\prod_{w\in V}\prod_{u\in U_{w,n}} \mathrm{rist}_G(u)$ for $n\ge \max_{w\in V}\{d(w)\}$. We consider now $V\cup\{v\}$ a larger set of non-comparable vertices. By the base case we obtain
\begin{align*}
    &\lim_{i\to \infty}\frac{\log|\mathrm{rist}_G(v)\big(\prod_{w\in V}\mathrm{rist}_G(w)\big)G_i:\big(\prod_{u\in U_{v,n}}\mathrm{rist}_G(u)\big)\big(\prod_{w\in V}\mathrm{rist}_G(w)\big)G_i|}{\log|G:G_i|}\\
    &\le \lim_{i\to \infty}\frac{\log|\mathrm{rist}_G(v)G_i:\big(\prod_{u\in U_{v,n}}\mathrm{rist}_G(u)\big)G_i|}{\log|G:G_i|}\\
    &= \lim_{i\to \infty}\frac{\log|\mathrm{rist}_G(v)G_i:G_i|}{\log|G:G_i|}-\lim_{i\to \infty}\frac{\log|\big(\prod_{u\in U_{v,n}}\mathrm{rist}_G(u)\big)G_i:G_i|}{\log|G:G_i|}\\
    &=0.
\end{align*}

Then, arguing as before, $\mathrm{rist}_G(v)\prod_{w\in V} \mathrm{rist}_G(w)$ has the same Hausdorff dimension as its subgroup $\prod_{w\in V\cup\{v\}}\prod_{u\in U_{w,n}} \mathrm{rist}_G(u)$ for $n\ge \max_{w\in V\cup\{v\}}\{d(w)\}$; proving the case when $V$ is finite.

Let us assume now that $V$ is countable but not finite. By construction, for every $n\geq 0$, we have  $H_{V_n}\le H_V$  and since each $V_n$ is finite 
    $$\mathrm{hdim}_{G}^\mathcal{N}(H_{V_n})=\mu_{\omega}(V_n).$$
    Then, by the monotonicity and countable stability of the Hausdorff dimension, we obtain that 
   \begin{align*}
       \mathrm{hdim}_{G}^\mathcal{N}(H_V)&\ge \mathrm{hdim}_{G}^{\mathcal{N}}(\bigcup_{n\ge 0}H_{V_n})= \sup_{n\ge 0} \mathrm{hdim}_G^\mathcal{N}(H_{V_n})=\sup_{n\ge 0}\mu_\omega(V_n)\\
       &=\lim_{n\to\infty}\mu_{\omega}(V_n)=\mu_{\omega}(V),
   \end{align*}
   as the sequence $\{\mu_\omega(V_n)\}_{n\ge 0}$ is non-decreasing.
   
   For the converse inequality, let us consider, for each $n\ge 0$, the
   subset of vertices
   $$W_n:=V_n\cup \{\text{all vertices at the $n$th level comparable to the vertices in $V\setminus V_n$}\},$$ and define 
   $$K_{n}:=H_{V_n}\prod_{v\in W_n\setminus V_n}\mathrm{rist}_G(v).$$ 
   Clearly $H_V\le K_n$ for every $n\ge 0$, and by the monotonicity of the Hausdorff dimension, we get
   \begin{align*}
       \mathrm{hdim}_{G}^\mathcal{N}(H_V)&\le  \inf_{n\ge 0} \mathrm{hdim}_{G}^\mathcal{N}(K_n) = \inf_{n\ge 0}\big( \mu_\omega(V_n)+\mu(W_n\setminus V_n) \big)\\
       &=\lim_{n\to\infty}\big(\mu_\omega(V_n)+\mu(W_n\setminus V_n)\big)=\mu_{\omega}(V)
   \end{align*}
   as the sequence $\{\mu_\omega(V_n)+\mu(W_n\setminus V_n)\}_{n\ge 0}$ is non-increasing.
\end{proof}

\subsection{Hausdorff spectrum and normal Hausdorff spectrum}

We generalize the result of Klopsch \cite[Chapter VII Theorem 1.6]{KlopschPhD} on the completeness of the Hausdorff spectrum of branch profinite groups both to any nice filtration and to a wider class of groups, including the closures of the main examples of weakly regular branch groups in the literature. We also use the reasoning in \cite[Theorem~7]{AbertVirag} to study the normal Hausdorff spectrum with respect to $\mathcal{S}$ of a wide class of self-similar groups. 

\begin{theorem}
\label{theorem: full spectra}
    Let $G\le \mathrm{Aut}~T_{\mathbf{m}}$ be a closed and level-transitive subgroup satisfying:
    \begin{enumerate}
        \item[$\mathrm{(i)}$] for all $n\ge 1$ we have $\mathrm{hdim}^\mathcal{N}_{G}(\mathrm{Rist}_G(n))=1$.
    \end{enumerate} 
    Then the Hausdorff spectrum of $G$ is given by
        $$\mathrm{hspec}^{\mathcal{N}}(G)=[0,1].$$
    Furthermore, if
    \begin{enumerate}
        \item[$\mathrm{(ii)}$] for all $n\ge 1$ we have $\mathrm{hdim}^\mathcal{N}_{G}(\mathrm{Rist}_G(n)')=1$;
    \end{enumerate}
    then the normal Hausdorff spectrum of $G$ is given by
    $$\mathrm{hspec}^{\mathcal{N}}_{\trianglelefteq}(G)=\{0,1\}.$$
\end{theorem}
\begin{proof}
    The result on the Hausdorff spectrum is a direct application of \textcolor{teal}{Lemmata}~\ref{lemma: dimension of product} and \ref{lemma: existence of V} and \cref{proposition: dimension is a series}. For the normal Hausdorff spectrum, \cref{lemma: normal subgroups contain derived of rist} implies that every non-trivial  normal closed subgroup $N$ of $G$ contains $\mathrm{Rist}_G(n)'$ for some integer $n\geq 1$. Hence, the result follows from assumption (ii) and the monotonicity of the Hausdorff dimension.
\end{proof}

The next proposition yields the main classes of groups satisfying assumptions (i) and (ii) in \cref{theorem: full spectra}:

\begin{proposition}
\label{corollary: full and normal spectra}
    Assumption $\mathrm{(i)}$ in \cref{theorem: full  spectra} is  satisfied by branch profinite groups for any nice filtration.
    
    Assumption $\mathrm{(ii)}$ in \cref{theorem: full  spectra} is  satisfied for $\mathcal{S}$ by any weakly branch group $G$ which is weakly regular branch over $K^{(i)}$ for some $i\ge 0$, where $K$ is a finite index subgroup of $G$ and satisfying any of the following additional conditions: 
    \begin{enumerate}
        \item[\normalfont(a)] $G$ is a subgroup of $W_p$;
        \item[\normalfont(b)] $G$ is finitely generated.
    \end{enumerate}
\end{proposition}
\begin{proof}
    The part of assumption (i) for branch groups follows from the rigid level stabilizers being of finite index and thus 1-dimensional in the group for any filtration. 
    
    For the case of weakly regular branch groups, by the first assumption every proper quotient of $G$ will be virtually solvable by \cref{lemma: normal subgroups contain derived of rist}. Thus both (a) or (b) imply that every non-trivial normal subgroup of $G$ has the same Hausdorff dimension as $G$ in $\mathrm{Aut}~T_m$ by \cite[Theorem~5]{AbertVirag} and \cite[Proposition 3.1]{JorgeMikel} respectively. By \cite[Proposition 5.3]{Jorge} and \cref{thm: Jorge_strong}, the group $G$ has positive and strong Hausdorff dimension in $\mathrm{Aut}~T_m$. Thus for any non-trivial normal subgroup $N\le G$ we obtain by \cref{lemma: strong dimension then product} that
    $$\mathrm{hdim}_G^\mathcal{S}(N)=\frac{\mathrm{hdim}_{\mathrm{Aut}~T_m}^\mathcal{S}(N)}{\mathrm{hdim}_{\mathrm{Aut}~T_m}^\mathcal{S}(G)}=1.$$
    In particular $\mathrm{hdim}_G^\mathcal{S}(\mathrm{Rist}_G(n)')=1$ for every $n\ge 1$.
\end{proof}

\cref{theorem: full spectra} together with \cref{corollary: full and normal spectra} gives the Hausdorff spectrum and the normal Hausdorff spectrum of the closures of the main examples of weakly regular branch groups: the Basilica group \cite{Gri-Zuk:Basilica} and its different generalizations to the $m$-adic tree \cite{pBasilica, GeneralizedBasilica}, the Brunner-Sidki-Vieira group \cite{BSV1} and its generalization to the $m$-adic tree \cite{BSV2} and the constant GGS-groups \cite{GGS} and its generalizations to the $p$-adic tree such as multi-EGS groups \cite{JoneAnitha}.

Abért and Virág proved that for a 1-dimensional level-transitive closed subgroup of $W_p$ its normal Hausdorff spectrum is $\{0,1\}$ and asked whether this can be extended to all positive dimensional subgroups of $W_p$. However, \cref{theorem: full  spectra} and \cref{corollary: full and normal spectra} together with the existence of zero-dimensional branch pro-$p$ subgroups of $W_p$ in \cite[Theorem C]{Jorge} suggest that this restriction on the normal Hausdorff spectrum does not come necessarily from the Hausdorff dimension of the group in~$W_p$. Therefore it would be interesting to see whether all weakly branch groups have the same Hausdorff spectrum and normal Hausdorff spectrum independently of their Hausdorff dimension: 

\begin{question}
    Is it true that $\mathrm{hspec}^\mathcal{S}(G)=[0,1]$ and $\mathrm{hspec}_{\trianglelefteq}^\mathcal{S}(G)=\{0,1\}$ for any weakly branch group $G$? And replacing $\mathcal{S}$ with any nice filtration~$\mathcal{N}$?
\end{question}

\section{Finitely generated branch groups}
\label{section: finitely generated Hausdorff spectrum}

In this section, we show that regular branch profinite groups have complete finitely generated Hausdorff spectrum. Following the notation of Siegenthaler in \cite{SiegenthalerBinary}, for $v\in T_{m}$ and $g\in \mathrm{Aut}~T_m$ we define the automorphism $g*v$ as the automorphism with section $g$ at $v$ and trivial label at every vertex not below $v$. For any $h\in \mathrm{Aut}~T_m$ it is immediate that
$$(g*v)^h=(g^{h|_v})*((v)h).$$
Similarly, for $v\in T_{m}$ and a subgroup $G\le \mathrm{Aut}~T_m$, we define the subgroup
$$G*v:=\{g*v\mid g\in G\} \le \mathrm{rist}(v)\le   \mathrm{Aut}~T_m.$$
Furthermore, for a non-comparable subset $V\subseteq T_m$, we define $G*V$ via the direct product
$$G*V:=\prod_{v\in V}G*v,$$
which is well defined as $[G*v,G*w]=G*v\cap G*w=1$ for every two distinct vertices $v,w\in V$. We say that a finite set of non-comparable vertices $Y\subseteq T_m$ is a \textit{transversal} if every vertex of $T_m$ is comparable to a unique vertex in $Y$. We may use the following easy but powerful lemma in our construction.

\begin{lemma}
\label{lemma: transitive action of branching subgroups}
    Let $G\le \mathrm{Aut}~T_m$ be a regular branch group branching over $K$. Then there exists some level $k\ge 0$ such that $K$ acts level-transitively on $T_v$ for every vertex $v$ at level $k$. Furthermore, if $K\le G\le W_p$, then $k\le \log_p|G:K|$.
\end{lemma}
\begin{proof}
    The first statement follows from the same proof as the one of \cite[Lemma~2.5]{Lee:maximal} considering the branching subgroup $K$ instead of a rigid level stabilizer. Thus, let us prove the second statement. First note that we may assume that both $K$ and $G$ are closed. Indeed, on the one hand, a group is level-transitive if and only if its closure is level-transitive and, on the other hand, we have $|G:K|\ge |\overline{G}:\overline{K}|$ by \cref{lemma: indices grow in closures}. 
    
    Now, for a regular branch profinite group~$G$ we can always find a branch structure over a level-stabilizer $\mathrm{St}_G(n)$ by \cite[Theorem~3]{SunicHausdorff}. Since $G$ is self-similar we may assume $\mathrm{St}_G(n)$ acts non-trivially on the $(n+1)$st level of the tree and if $G\le W_p$ this implies $\mathrm{St}_G(n)$ has an element acting transitively on the immediate descendants of a vertex at the $n$th level. Furthermore, since $\mathrm{St}_G(n)$ is normal in $G$ and $G$ is level-transitive, this holds for any vertex at the $n$th level. Then since $\mathrm{St}_G(n)$ is branching we can replicate this element in all the descendants of a given vertex at the $n$th level and conclude that $\mathrm{St}_G(n)$ acts level-transitively on $T_m^{[v]}$ for every vertex $v$ at the $n$th level. By \cite[Theorem A]{JorgeEntropy}, we have that the number $n$ is at most the number of prime factors with repetition of the subgroup index $|G:K|$. Since $K\le G\le W_p$, this is precisely $\log_p|G:K|$, which yields the upper-bound on the statement.
\end{proof}

Let $G$ be a finitely generated branch group and $K$ a finite index branching subgroup of $G$. We say that an integer $k$ is \textit{$(G,K)$-bounded} if it can be bounded in terms of a function depending only on $d(G)$ and $|G:K|$. Then \cref{lemma: transitive action of branching subgroups} tells that the minimal level $k$ where $K$ acts transitively on the subtrees rooted at this level is $(G,K)$-bounded.

We may assume in the following that the branching subgroup of a regular branch group is normal. Indeed, the following holds.

\begin{lemma}
\label{lemma: normal core}
    Let $G\le \mathrm{Aut}~T_m$ be a regular branch group branching over a subgroup~$K$. Then $G$ also branches over $\mathrm{Core}_G(K)$.
\end{lemma}
\begin{proof}
    Since $K$ is of finite index in $G$ its normal core is also of finite index in $G$. Therefore it is enough to prove that $\mathrm{Core}_G(K)$ is also branching. Since the subgroup $\psi^{-1}(\mathrm{Core}_G(K)\times\overset{m}{\ldots}\times \mathrm{Core}_G(K))\le K$ is normal in $G$ it must be contained in $\mathrm{Core}_G(K)$ so $\mathrm{Core}_G(K)$ is branching.
\end{proof}

Note that in the general case of weakly regular branch groups,  since the branching subgroup $K$ is not of finite index in $G$, we cannot conclude that $\mathrm{Core}_G(K)\ne 1$. It would be interesting to see whether \cref{lemma: normal core} still holds in this case:

\begin{question}
    Does every weakly regular branch group contain a non-trivial normal branching subgroup?
\end{question}

Note that if the group is fractal, that is, if $\varphi_v(\st_G(v))=G$ for every $v\in T_m$, by \cite[Lemma 1.3]{Bartholdi-Siegenthaler-Zalesskii} any weakly regular branch group is weakly regular branch over a non-trivial normal subgroup.

In the upcoming result, we construct explicitly an element of infinite order in  a non-trivial branching closed subgroup of $\mathrm{Aut}~T_m$. The result itself is not new: by an astonishing result of Zelmanov, periodic pro-$p$ groups are locally finite; see \cite[Theorem~1]{ZelmanovPro-p}. What is more, periodic profinite groups are locally finite by Wilson's reduction to pro-$p$ groups; see \cite[Theorem 1]{WilsonBurnside}. In turn, every infinite finitely generated profinite group has an element of infinite order. Furthermore, for branch pro-$p$ groups elements of infinite order are abundant; see \cite[Theorem 1]{TorsionBranch}. Nonetheless, the above results only ensure the existence of such an element. We avoid using the aforementioned deep results, and instead, construct an element of infinite order for branching profinite groups that are not necessarily topologically finitely generated. Before stating the result, we fix the following notation: for an infinite collection of (non-necessarily distinct) vertices $\{v_n\}_{n\ge 1}\subseteq T_m$, we write $v_1\dotsb v_n\dotsb $ for the infinite path in $T_m$ corresponding to the infinite word formed by concatenating all the finite words $\{v_n\}_{n\ge 1}$.

\begin{lemma}
    \label{lemma: branching has element of infinite order}
    Let $K\le \mathrm{Aut}~T_m$ be a non-trivial branching closed subgroup. Then~$K$ contains an element $g$ of infinite order.
\end{lemma}
\begin{proof}
    Since $K$ is non-trivial, there exists some $h\in K$ such that $h$ acts non-trivially on some vertex $v_1$ at some level $k$ of $T_m$. If $h$ is of infinite order we are done. Otherwise, let $t_1$ be the length of the $h$-orbit of~$v_1$ on the $k$th level of the tree. Then $h^{t_1}\in \mathrm{st}(v_1)$ and the order of $h$ is divisible by $t_1$. Now let us write
    $$\psi_k(h)=(h_1,\dotsc, h_{m^k})\sigma$$
    for some $\sigma\in \mathrm{Sym}(m)\wr\overset{k}{\ldots}\wr \mathrm{Sym}(m)\le \mathrm{Sym}(m^k)$. Then
    $$h^{t_1}|_{v_1}=h_{j_1}\dotsb h_{j_{t_1}}$$
    for some pairwise distinct $1\le j_1,\dotsc,j_{t_1}\le m^k$ by our choice of $v_1$ and $t_1$. If $h_{j_1}\dotsb h_{j_{t_1}}$ is non-trivial, then there exists $v_2$ such that the $h$-orbit of $v_1v_2$ has length $t_1t_2>t_1$ and $h$ has order divisible by $t_1t_2$. If $h_{j_1}\dotsb h_{j_{t_1}}$ is trivial we consider $h(h*v_1)$, for which $(h(h*v_1))^{t_1}\in \mathrm{st}(v_1)$ too and
    $$(h(h*v_1))^{t_1}|_{v_1}= h_{j_1}\dotsb h_{j_{t_1}}h=h\ne 1$$
    as $h*v_1\in \mathrm{St}(k)$. Then there exists $v_2$ such that the $h(h*v_1)$-orbit of $v_1v_2$ has length $t_1t_2>t_1$ and $h(h*v_1)$ has order divisible by $t_1t_2>t_1$. Since $K$ is closed in $\mathrm{Aut}~T_m$, we may define inductively
    $$g:=h\prod_{i\ge 1} (h*v_i)^{\epsilon_i},$$
    where each $\epsilon_i$ is chosen either 1 or 0 arguing as before for each $i\ge 1$.  Since $K$ is branching, $g\in K$ and it satisfies the properties in the statement.
\end{proof}

Now we are in position to provide the key construction on which our main results are based:

\begin{lemma}
\label{lemma: construction of LI}
    Let $G\le \mathrm{Aut}~T_m$ be a weakly regular branch group. Let $K\trianglelefteq G$ be a finitely generated normal branching subgroup and $X=\{x_n\}_{n\ge 1}\subseteq T_m$ a subset of non-comparable vertices. Then there exists a finitely generated subgroup $L_{X}\le \overline{K}$ such that
    \begin{equation}\label{eq: propertiesLX}\prod_{n\ge 1} \overline{K'}*x_n Y_n\le \overline{L_X'}\le  \overline{L}_X\le  \overline{K}*X,
    \end{equation}
    where for each $n\ge 1$, the subset $Y_n$ is a transversal of $T_m$. Furthermore $d(L_X)$ is the same for every choice of $X\subseteq T_m$. If $G\le W_p$ and it is finitely generated and regular branch over $K$, then $d(L_X)$ is $(G,K)$-bounded.
\end{lemma}
\begin{proof}
    Let $K=\langle k_1,\dotsc, k_r\rangle$, where $r:=d(K)<\infty$. Let $g\in \overline{K}$ be an element obtained via the construction in \cref{lemma: branching has element of infinite order}. Since~$g$ is of infinite order, there exists an infinite path $v_1\dotsb v_n\dotsb$ in $T_m$ such that the section $g^{t_1\dotsb t_i}|_{v_1\dotsb v_i}\in \mathrm{st}_K(v_1\dotsb v_i)$ is of infinite order.  Moreover, if for every $i\ge 1$, $t_1\dotsb t_i$ is the length of the $g$-orbit of the vertex $v_1\dotsb v_i$, we have that $t_i\ge r$, for every $i\ge 1$.  
    Then, for each $i\ge 1$, we set $v_{i,1}:=v_1\dotsb v_i$ and $v_{i,2},\dotsc,v_{i,t_i}$ to be the remaining vertices in the $g^{t_1\dotsb t_{i-1}}$-orbit of $v_{i,1}$, where for $i=1$ we simply consider the $g$-orbit of $v_1$. More explicitly,
    $$v_{i,j}=(v_{i,1})g^{t_1\dotsb t_{i-1}(j-1)}$$
    for every $1\le j\le t_i$. Consider the following  sets of vertices
    $$V_i:=\bigcup_{j=1}^{t_i}v_{i,j}.$$
    Note that $(v_{i+1,j})g^{t_1\dotsb t_{i-1}y}\ne v_{i+1,k}$ for $1\le j,k\le t_{i+1}$ such that $k\ne j$ and any $1\le y\le t_i-1$. Indeed, we have $(v_{i,1})g^{t_1\dotsb t_{i-1}y}=v_{i,1+y}\ne v_{i,1}$ and both $v_{i+1,k},v_{i+1,j}$ are descendants of $v_{i,1}$.

     Now, by \cref{lemma: transitive action of branching subgroups} there exists a level in $T_m$ such that $K$ acts transitively on all the subtrees hanging from that level. Let us denote by $W$ the finite set of all vertices at that level of $T_m$. As $\overline{K}$ is branching, for each $w\in W$ we define the automorphism $a_w\in \overline{K}$ as
    $$a_w:=\prod_{n\ge 1} g*x_nw,$$
    and similarly, for each $1\le j\le r$, $b_{j,w},c_{j,w}\in \overline{K}$ as
    $$ b_{j,w}:=\prod_{n\ge 1}k_j^{{(a_w^{t_1\dotsb t_{n-1}(r-j)}|_{x_nwv_{n,j}})}^{-1}}*x_nwv_{n,j}$$
    and 
    $$ c_{j,w}:=\prod_{n\ge 1}k_j*x_nwv_{n,r}.$$
    Note that $b_{j,w}\in \overline{K}$ because $\overline{K}$ is normal in $\overline{G}$, and $\overline{G}$ is self-similar. Lastly, for $1\le j\le r$, we define $a_j\in \overline{K}$ as
    $$a_j:=\prod_{n\ge 1} k_j*x_n.$$
    We now consider the finitely generated subgroup 
    $$L_X:=\langle a_w,b_{j,w}, c_{j,w}, a_j\mid 1\le j\le r\text{ and }w\in W\rangle\le \overline{K}.$$
    Note that if $G\le W_p$ and it is finitely generated regular branch over $K$, then the cardinality of $W$ is $(G,K)$-bounded by \cref{lemma: transitive action of branching subgroups}. In addition $r$ is $(G,K)$-bounded  $r\le |G:K|(d(G)-1)+1$ by the Schreier-index formula. Thus $d(L_X)$ is $(G,K)$-bounded in this case. In any case $d(L_X)$ is the same for any choice of $X$ by construction.
    
    Now, for any $1\le j<s\le r$, any $w\in W$ and any $u\ge 1$, we get
    \begin{align}
        \label{align: commutator of bs}
        [b_{j,w}^{a_w^{t_1\dotsb t_{u-1}(r-j)}},b_{s,w}^{a_w^{t_1\dotsb t_{u-1}(r-s)}}]=[k_j,k_s]*x_uwv_{u,r}.
    \end{align}

    Hence, any $k\in K$ can be written as a word on the generators of $K$ as $k=k_{i_1}\dotsb k_{i\eta}$, with $1\le i_1,\dotsc,i_\eta\le r$. Thus, conjugating the element in \cref{align: commutator of bs} with $c_{i_1,w}\dotsb c_{i_\eta,w}$ yields that
    $$[k_j,k_s]^{k}*x_uwv_{u,r}\in L_X'.$$
    Thus, for every $u\ge 1$,
    $$K'*x_uwv_{u,r}\le L_X'$$
     holds as $K'=\langle [k_j,k_s]\mid 0\le j,s\le r\rangle^{K}$. Since $K'$ is normal in $G$, using the elements $a_j$ we can move the above subgroup to every vertex at the $K$-orbit of $wv_{u,r}$. Note that for each $w\in W$, $K$ acts level-transitively on  $T^{[w]}_m$, and so, for every $u\ge 1$, we obtain that
    $$K'*x_u Y_u\le L_X',$$
    where $Y_u$ is a transversal of $T_m$. Thus, taking closures yields
    $$\prod_{n\ge 1} \overline{K'}*x_n Y_n\le \overline{L_X'}.$$

    The containment $\overline{L}_X\le \overline{K}*X$ follows from the generators of $L_X$ being contained in $\overline{K}*X$ and the latter being a closed subgroup.
\end{proof}

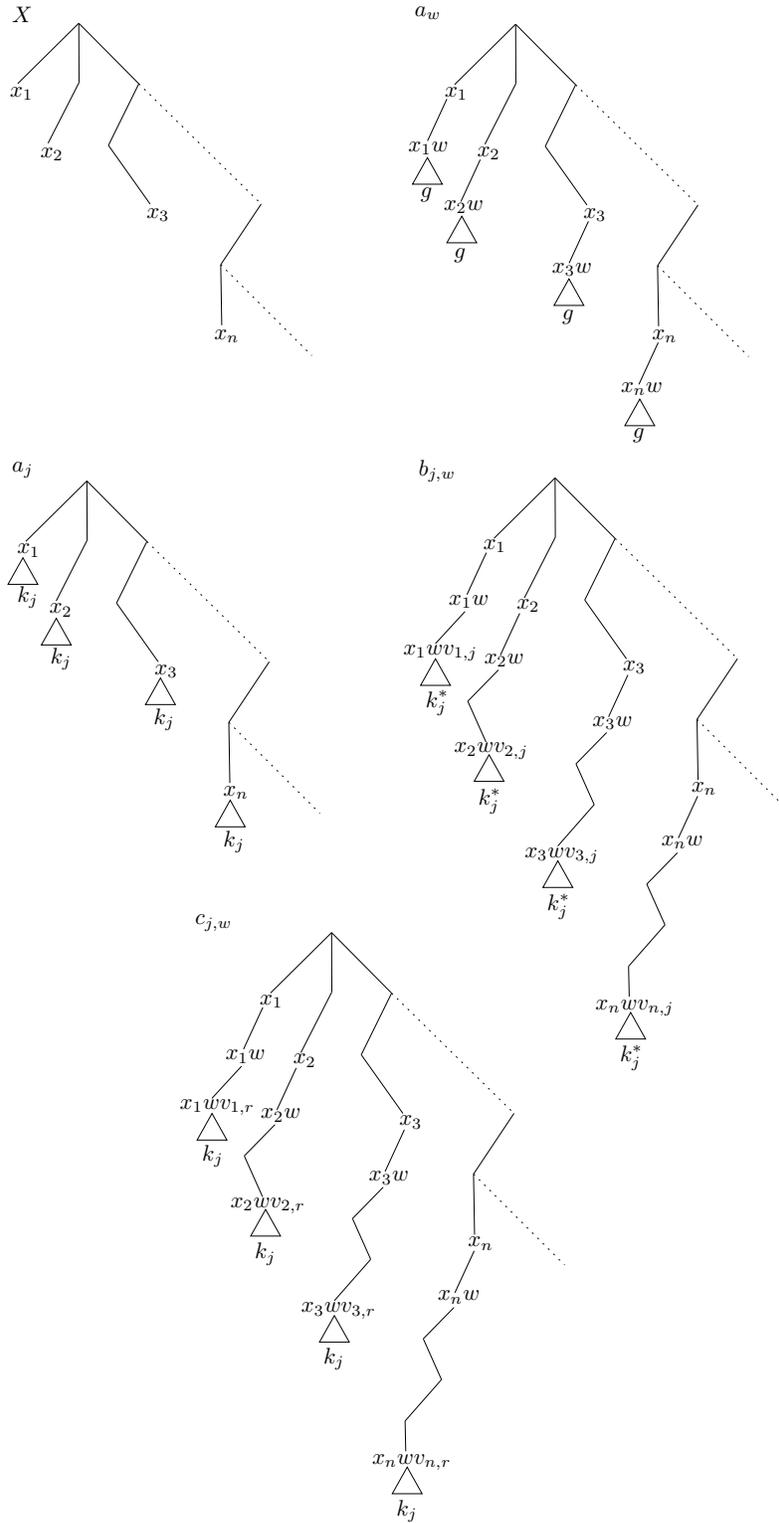
\begin{figure}
    \centering
    \scalebox{.85}{
    \begin{tikzpicture}[x=0.75pt,y=0.75pt,yscale=-0.6,xscale=0.6]

\draw    (71,21) -- (10.62,80.76) ;
\draw    (71,21) -- (71.24,79.73) ;
\draw    (71.24,79.73) -- (40.62,139.76) ;
\draw    (71,21) -- (130.37,80.9) ;
\draw    (130.37,80.9) -- (100.34,141.87) ;
\draw    (251.37,199.9) -- (211.37,259.9) ;
\draw    (100.34,141.87) -- (141.56,199.97) ;
\draw  [dash pattern={on 0.84pt off 2.51pt}]  (130.37,80.9) -- (251.37,199.9) ;
\draw    (211.37,259.9) -- (212.37,319.9) ;
\draw  [dash pattern={on 0.84pt off 2.51pt}]  (211.37,259.9) -- (301.37,349.9) ;


\draw    (503,22) -- (442.62,81.76) ;
\draw    (503,22) -- (503.24,80.73) ;
\draw    (503.24,80.73) -- (472.62,140.76) ;
\draw    (503,22) -- (562.37,81.9) ;
\draw    (562.37,81.9) -- (532.34,142.87) ;
\draw    (683.37,200.9) -- (643.37,260.9) ;
\draw    (532.34,142.87) -- (573.56,200.97) ;
\draw  [dash pattern={on 0.84pt off 2.51pt}]  (562.37,81.9) -- (683.37,200.9) ;
\draw    (643.37,260.9) -- (644.37,320.9) ;
\draw  [dash pattern={on 0.84pt off 2.51pt}]  (643.37,260.9) -- (733.37,350.9) ;

\draw    (435,96) -- (415.89,137.42) ;
\draw    (468,156) -- (448.89,197.42) ;
\draw    (575,217) -- (555.89,258.42) ;
\draw    (644,337) -- (624.89,378.42) ;
\draw   (415.74,154) -- (430.48,180.3) -- (401,180.3) -- cycle ;
\draw   (449.74,212) -- (464.48,238.3) -- (435,238.3) -- cycle ;
\draw   (555.74,274) -- (570.48,300.3) -- (541,300.3) -- cycle ;
\draw   (625.74,393) -- (640.48,419.3) -- (611,419.3) -- cycle ;


\draw   (15.74,550) -- (30.48,576.3) -- (1,576.3) -- cycle ;
\draw    (79,474) -- (18.62,533.76) ;
\draw    (79,474) -- (79.24,532.73) ;
\draw    (79.24,532.73) -- (48.62,592.76) ;
\draw    (79,474) -- (138.37,533.9) ;
\draw    (138.37,533.9) -- (108.34,594.87) ;
\draw    (259.37,652.9) -- (219.37,712.9) ;
\draw    (108.34,594.87) -- (149.56,652.97) ;
\draw  [dash pattern={on 0.84pt off 2.51pt}]  (138.37,533.9) -- (259.37,652.9) ;
\draw    (219.37,712.9) -- (220.37,772.9) ;
\draw  [dash pattern={on 0.84pt off 2.51pt}]  (219.37,712.9) -- (309.37,802.9) ;

\draw   (48.74,610) -- (63.48,636.3) -- (34,636.3) -- cycle ;
\draw   (151.74,669) -- (166.48,695.3) -- (137,695.3) -- cycle ;
\draw   (220.74,790) -- (235.48,816.3) -- (206,816.3) -- cycle ;


\draw    (542,471) -- (481.62,530.76) ;
\draw    (542,471) -- (542.24,529.73) ;
\draw    (542.24,529.73) -- (511.62,589.76) ;
\draw    (542,471) -- (601.37,530.9) ;
\draw    (601.37,530.9) -- (571.34,591.87) ;
\draw    (722.37,649.9) -- (682.37,709.9) ;
\draw    (571.34,591.87) -- (612.56,649.97) ;
\draw  [dash pattern={on 0.84pt off 2.51pt}]  (601.37,530.9) -- (722.37,649.9) ;
\draw    (682.37,709.9) -- (683.37,769.9) ;
\draw  [dash pattern={on 0.84pt off 2.51pt}]  (682.37,709.9) -- (772.37,799.9) ;

\draw    (614,666) -- (594.89,707.42) ;
\draw    (683,786) -- (663.89,827.42) ;
\draw   (544.74,850) -- (559.48,876.3) -- (530,876.3) -- cycle ;
\draw   (616.74,1000) -- (631.48,1026.3) -- (602,1026.3) -- cycle ;
\draw    (580.67,794.33) -- (544.67,835.33) ;
\draw    (614.67,954.33) -- (615.33,984.33) ;
\draw    (593,723) -- (562.67,754) ;
\draw    (562.67,754) -- (580.67,794.33) ;
\draw    (650.67,913.33) -- (614.67,954.33) ;
\draw    (663,842) -- (632.67,873) ;
\draw    (632.67,873) -- (650.67,913.33) ;

\draw    (474,545) -- (454.89,586.42) ;
\draw    (507,605) -- (487.89,646.42) ;
\draw   (423.74,650) -- (438.48,676.3) -- (409,676.3) -- cycle ;
\draw    (452.91,603.23) -- (423.67,635) ;
\draw    (486,661) -- (455.67,692) ;
\draw    (455.67,692) -- (473.67,732.33) ;

\draw   (476.74,745) -- (491.48,771.3) -- (462,771.3) -- cycle ;


\draw    (321,921) -- (260.62,980.76) ;
\draw    (321,921) -- (321.24,979.73) ;
\draw    (321.24,979.73) -- (290.62,1039.76) ;
\draw    (321,921) -- (380.37,980.9) ;
\draw    (380.37,980.9) -- (350.34,1041.87) ;
\draw    (501.37,1099.9) -- (461.37,1159.9) ;
\draw    (350.34,1041.87) -- (391.56,1099.97) ;
\draw  [dash pattern={on 0.84pt off 2.51pt}]  (380.37,980.9) -- (501.37,1099.9) ;
\draw    (461.37,1159.9) -- (462.37,1219.9) ;
\draw  [dash pattern={on 0.84pt off 2.51pt}]  (461.37,1159.9) -- (551.37,1249.9) ;

\draw    (393,1116) -- (373.89,1157.42) ;
\draw    (462,1236) -- (442.89,1277.42) ;
\draw   (323.74,1300) -- (338.48,1326.3) -- (309,1326.3) -- cycle ;
\draw   (395.74,1450) -- (410.48,1476.3) -- (381,1476.3) -- cycle ;
\draw    (359.67,1244.33) -- (323.67,1285.33) ;
\draw    (393.67,1404.33) -- (394.33,1434.33) ;
\draw    (372,1173) -- (341.67,1204) ;
\draw    (341.67,1204) -- (359.67,1244.33) ;
\draw    (429.67,1363.33) -- (393.67,1404.33) ;
\draw    (442,1292) -- (411.67,1323) ;
\draw    (411.67,1323) -- (429.67,1363.33) ;

\draw    (253,995) -- (233.89,1036.42) ;
\draw    (286,1055) -- (266.89,1096.42) ;
\draw   (202.74,1100) -- (217.48,1126.3) -- (188,1126.3) -- cycle ;
\draw    (231.91,1053.23) -- (202.67,1085) ;
\draw    (265,1111) -- (234.67,1142) ;
\draw    (234.67,1142) -- (252.67,1182.33) ;

\draw   (255.74,1195) -- (270.48,1221.3) -- (241,1221.3) -- cycle ;



\draw (3,3) node [anchor=north west][inner sep=0.75pt]   [align=left] {$\displaystyle X$};
\draw (1,82) node [anchor=north west][inner sep=0.75pt]   [align=left] {$\displaystyle x_{1}$$ $$ $};
\draw (31,142) node [anchor=north west][inner sep=0.75pt]   [align=left] {$\displaystyle x_{2}$};
\draw (136,202) node [anchor=north west][inner sep=0.75pt]   [align=left] {$\displaystyle x_{3}$};
\draw (203,322) node [anchor=north west][inner sep=0.75pt]   [align=left] {$\displaystyle x_{n}$};

\draw (636,322) node [anchor=north west][inner sep=0.75pt]   [align=left] {$\displaystyle x_{n}$};
\draw (568,202) node [anchor=north west][inner sep=0.75pt]   [align=left] {$\displaystyle x_{3}$};
\draw (463,141) node [anchor=north west][inner sep=0.75pt]   [align=left] {$\displaystyle x_{2}$};
\draw (431,82) node [anchor=north west][inner sep=0.75pt]   [align=left] {$\displaystyle x_{1}$};

\draw (401,2) node [anchor=north west][inner sep=0.75pt]   [align=left] {$\displaystyle a_{w}$};

\draw (395,136) node [anchor=north west][inner sep=0.75pt]   [align=left] {$\displaystyle x_{1}w$};
\draw (407,181) node [anchor=north west][inner sep=0.75pt]   [align=left] {$\displaystyle g$};
\draw (430,194) node [anchor=north west][inner sep=0.75pt]   [align=left] {$\displaystyle x_{2}w$};
\draw (440,241) node [anchor=north west][inner sep=0.75pt]   [align=left] {$\displaystyle g$};
\draw (537,256) node [anchor=north west][inner sep=0.75pt]   [align=left] {$\displaystyle x_{3}w$};
\draw (547,302) node [anchor=north west][inner sep=0.75pt]   [align=left] {$\displaystyle g$};
\draw (605,374) node [anchor=north west][inner sep=0.75pt]   [align=left] {$\displaystyle x_{n}w$};
\draw (617,420) node [anchor=north west][inner sep=0.75pt]   [align=left] {$\displaystyle g$};


\draw (212,774) node [anchor=north west][inner sep=0.75pt]   [align=left] {$\displaystyle x_{n}$};
\draw (144,654) node [anchor=north west][inner sep=0.75pt]   [align=left] {$\displaystyle x_{3}$};
\draw (40,593) node [anchor=north west][inner sep=0.75pt]   [align=left] {$\displaystyle x_{2}$};
\draw (8,534) node [anchor=north west][inner sep=0.75pt]   [align=left] {$\displaystyle x_{1}$};
\draw (3,453) node [anchor=north west][inner sep=0.75pt]   [align=left] {$\displaystyle a_{j}$};

\draw (8,576) node [anchor=north west][inner sep=0.75pt]   [align=left] {$\displaystyle k_{j}$};
\draw (41,637) node [anchor=north west][inner sep=0.75pt]   [align=left] {$\displaystyle k_{j}$};
\draw (143,697) node [anchor=north west][inner sep=0.75pt]   [align=left] {$\displaystyle k_{j}$};
\draw (212,818) node [anchor=north west][inner sep=0.75pt]   [align=left] {$\displaystyle k_{j}$};


\draw (675,770) node [anchor=north west][inner sep=0.75pt]   [align=left] {$\displaystyle x_{n}$};
\draw (607,650) node [anchor=north west][inner sep=0.75pt]   [align=left] {$\displaystyle x_{3}$};
\draw (502,589) node [anchor=north west][inner sep=0.75pt]   [align=left] {$\displaystyle x_{2}$};
\draw (469,530) node [anchor=north west][inner sep=0.75pt]   [align=left] {$\displaystyle x_{1}$};

\draw (405,451) node [anchor=north west][inner sep=0.75pt]   [align=left] {$\displaystyle b_{j,w}$};
\draw (577,705) node [anchor=north west][inner sep=0.75pt]   [align=left] {$\displaystyle x_{3} w$};
\draw (645,823) node [anchor=north west][inner sep=0.75pt]   [align=left] {$\displaystyle x_{n} w$};
\draw (510,835) node [anchor=north west][inner sep=0.75pt]   [align=left] {$\displaystyle x_{3} wv_{3,j}$};
\draw (532,880) node [anchor=north west][inner sep=0.75pt]   [align=left] {$\displaystyle k_j^*$};
\draw (580,985) node [anchor=north west][inner sep=0.75pt]   [align=left] {$\displaystyle x_{n} wv_{n,j}$};
\draw (604,1030) node [anchor=north west][inner sep=0.75pt]   [align=left] {$\displaystyle k_j^*$};
\draw (435,585) node [anchor=north west][inner sep=0.75pt]   [align=left] {$\displaystyle x_{1} w$};
\draw (470,643) node [anchor=north west][inner sep=0.75pt]   [align=left] {$\displaystyle x_{2} w$};
\draw (390,634) node [anchor=north west][inner sep=0.75pt]   [align=left] {$\displaystyle x_{1} wv_{1,j}$};
\draw (412,679) node [anchor=north west][inner sep=0.75pt]   [align=left] {$\displaystyle k_j^*$};
\draw (440,731) node [anchor=north west][inner sep=0.75pt]   [align=left] {$\displaystyle x_{2} wv_{2,j}$};
\draw (464,776) node [anchor=north west][inner sep=0.75pt]   [align=left] {$\displaystyle k_j^*$};


\draw (184,901) node [anchor=north west][inner sep=0.75pt]   [align=left] {$\displaystyle c_{j,w}$};
\draw (356,1155) node [anchor=north west][inner sep=0.75pt]   [align=left] {$\displaystyle x_{3} w$};
\draw (424,1273) node [anchor=north west][inner sep=0.75pt]   [align=left] {$\displaystyle x_{n} w$};
\draw (289,1285) node [anchor=north west][inner sep=0.75pt]   [align=left] {$\displaystyle x_{3} wv_{3,r}$};
\draw (311,1330) node [anchor=north west][inner sep=0.75pt]   [align=left] {$\displaystyle k_j$};
\draw (359,1435) node [anchor=north west][inner sep=0.75pt]   [align=left] {$\displaystyle x_{n} wv_{n,r}$};
\draw (383,1480) node [anchor=north west][inner sep=0.75pt]   [align=left] {$\displaystyle k_j$};
\draw (214,1035) node [anchor=north west][inner sep=0.75pt]   [align=left] {$\displaystyle x_{1} w$};
\draw (249,1093) node [anchor=north west][inner sep=0.75pt]   [align=left] {$\displaystyle x_{2} w$};
\draw (169,1084) node [anchor=north west][inner sep=0.75pt]   [align=left] {$\displaystyle x_{1} wv_{1,r}$};
\draw (191,1129) node [anchor=north west][inner sep=0.75pt]   [align=left] {$\displaystyle k_j$};
\draw (219,1181) node [anchor=north west][inner sep=0.75pt]   [align=left] {$\displaystyle x_{2} wv_{2,r}$};
\draw (243,1226) node [anchor=north west][inner sep=0.75pt]   [align=left] {$\displaystyle k_j$};
\draw (454,1220) node [anchor=north west][inner sep=0.75pt]   [align=left] {$\displaystyle x_{n}$};
\draw (386,1100) node [anchor=north west][inner sep=0.75pt]   [align=left] {$\displaystyle x_{3}$};
\draw (281,1039) node [anchor=north west][inner sep=0.75pt]   [align=left] {$\displaystyle x_{2}$};
\draw (248,980) node [anchor=north west][inner sep=0.75pt]   [align=left] {$\displaystyle x_{1}$};

\end{tikzpicture}   
    }
    \caption{The generators of $L_X$ for $X=\{x_n\}_{n\ge 1}\subseteq T_m$, where $k_j^*$ represents the appropriate conjugate of $k_j$ appearing in the proof of \cref{lemma: construction of LI}. }
    \label{fig:enter-label}
\end{figure}

\begin{proof}[Proof of \cref{theorem: finitely generated spectrum of regular branch}]

Let $G$ be a finitely generated regular branch group. Then by \cref{lemma: normal core}, $G$ has a normal branching subgroup $K$ of finite index. Moreover, by \cref{lemma: construction of LI}, there exists a subgroup $\overline{L_X}\le \overline{K}$ satisfying \eqref{eq: propertiesLX} and we have    \begin{align*}
        \{\mathrm{hdim}^{\mathcal{S}}_{\overline{G}}(\overline{L}_X)\mid X\subseteq T_m\text{ a subset of non-comparable vertices}\}&\subseteq \mathrm{hspec}^{\mathcal{S}}_{d(L_X)}(\overline{G})\\
        &\subseteq [0,1].
    \end{align*}
    We will show that the hypothesis of \cref{proposition: dimension is a series} are satisfied. In this way, we have that $\mathrm{hdim}^{\mathcal{S}}_{\overline{G}}(\overline{L}_X)=\mu(X)$ and we conclude that there exists $b\ge 2$ such that every number $\gamma\in [0,1]$ is realizable as the Hausdorff dimension of the closure of a $b$-generated subgroup $L_X\le \overline{K}\le\overline{G}$ by simply choosing $X\subseteq T_m$ such that $\mu(X)=\gamma$ (see \cref{lemma: existence of V}).

    To that aim it suffices to show that $\overline{K'}$ is 1-dimensional in $G$. Since $K$ is of finite index in $G$, \cref{lemma: indices grow in closures} implies that  $\overline{K}$ is 1-dimensional in $\overline{G}$. Furthermore, $K$ being finitely generated yields that  $K$ and $K'$ have the same Hausdorff dimension in $\Aut T_m$ (see \cite[Proposition~3.1]{JorgeMikel}). As $G$ is regular branch it has positive and strong Hausdorff dimension in $\Aut T_m$, by \cref{thm: Jorge_strong}. Then arguing as in the proof of \cref{corollary: full and normal spectra}, we obtain that $\overline{K'}$ is also 1-dimensional in $\overline{G}$. This finishes the proof.

If $G$ is just infinite, $K'$ has finite index in $G$ because $K'$ is normal in $G$ and by  \cref{lemma: normal subgroups contain derived of rist}, $G$ cannot be virtually abelian. Then, by \cref{lemma: indices grow in closures}, $\overline{K'}$ has dimension 1 in $\overline{G}$ (with respect to any filtration). The result holds by applying \cref{proposition: dimension is a series} which works for any nice filtration.
    
    If $G\le W_p$ we may further assume $d(L_X)$ is $(G,K)$-bounded by \cref{lemma: construction of LI}.
\end{proof}

\section{Iterated wreath products}
\label{section: iterated wreath products}

The goal of this section is to establish \cref{Theorem: WH main result}. To that aim, it suffices to prove the following result:

\begin{theorem}
\label{theorem: existence of 1-dimensional subgroups in WH}
    Let $H\le \mathrm{Sym}(m)$. Then $W_H$ contains a subgroup $K$ which is $(d(H)+1)$-generated and whose closure is 1-dimensional in $W_H$.
\end{theorem}

Since $W_H$ is regular branch over itself, the same construction as in the proof of \cref{theorem: finitely generated spectrum of regular branch} can be carried. Note also that the subgroup $K$ need not be normal in $W_H$. In fact, contrary to the more general case of a regular branch group $G$, in an iterated wreath product $W_H$ we may simplify the definition of the automorphisms $a_w$'s in the proof of \cref{theorem: finitely generated spectrum of regular branch} by considering a truncated version of the element $g$, so that the sections of $g$ appearing conjugating~$k_j$ in the definition of the $b_{j,w}$'s become trivial.

The proof of \cref{theorem: existence of 1-dimensional subgroups in WH} is divided into two steps. First we construct a family of groups generalizing \v{S}uni\'{c} groups in \cite{SunicHausdorff}:

\begin{lemma}
\label{lemma: existence of groups tending to 1}
    Let $H\le \mathrm{Sym}(m)$ be any transitive group where $m\ge 2$. Then there exist finitely generated, level-transitive, self-similar subgroups $G_1,G_2,\dotsc\le W_H$ such that $\lim_{r\to\infty }\mathrm{hdim}_{W_H}^{\mathcal{S}}(\overline{G}_r)=1$ and $d(G_r)$ is linear on $r$.
\end{lemma}

Then we generalize the argument of Siegenthaler in \cite{SiegenthalerBinary} to construct a 1-dimensional closed subgroup from this family of the generalized \v{S}uni\'{c} groups.

\begin{lemma}
\label{lemma: Siegenthaler generalization}
   Let $G_1,G_2,\dotsc\le W_H$ be a sequence of level-transitive subgroups such that $\lim_{r\to\infty }\mathrm{hdim}_{W_H}^{\mathcal{S}}(\overline{G}_r)=1$ and $d(G_r)$ is polynomial on $r$. Then there exists a group $K\le W_H$ generated by $d(H)+1$ elements whose closure in $W_H$ has Hausdorff dimension 1. 
\end{lemma}

Now \cref{theorem: existence of 1-dimensional subgroups in WH} is a direct consequence of both \cref{lemma: existence of groups tending to 1} and \cref{lemma: Siegenthaler generalization}.

\subsection{Generalized \v{S}uni\'{c} groups}

Let us fix $m\ge 3$ (the case $m=2$ reduces to the construction of \v{S}uni\'{c} for $p=2$ in \cite{SunicHausdorff}) and $H\le \mathrm{Sym}(m)$ transitive. We also fix some generating set $H=\langle X\rangle =\langle x_1,\dotsc,x_\ell \rangle$ where $\ell \le m-1$. Note this is possible as $d(H)\le m-1$ by \cite[Theorem 1.1]{TraceyGeneration}. Let $r\ge 1$ be an integer. For $1\le i\le \ell$ we define the subgroup
$$B_{i}:=\langle b_{i,1},\dotsc,b_{i,r}\rangle,$$
where the generators $b_{i,1},\dotsc, b_{i,r}$ are defined recursively via
\begin{align*}
    \psi(b_{i,1})&=(1,\dotsc,1,b_{i,2})\\
    &\,\,\, \vdots\\
    \psi(b_{i,r-1})&=(1,\dotsc,1,b_{i,r})\\
    \psi(b_{i,r})&=(1,\overset{i-1}{\dotsc},1,x_i,1,\dotsc,1,b_{i,1}).
\end{align*}
Then we define the group $G_{H,r}:=\langle H, B_{i}\mid 1\le i\le \ell\rangle$, where abusing the notation by $H$ we refer to the subgroup of rooted automorphisms that act as the elements of $H$ on the first level $\{1,\dots,m\}$. The group $G_{H,r}$ can be seen as a generalization of spinal groups where each generator $x_i\in X$ appears in a different spine. The group $G_{H,r}$ satisfies similar properties to the ones satisfied by the \v{S}uni\'{c} groups in \cite{SunicHausdorff}. We shall only focus on the properties needed to prove \cref{lemma: existence of groups tending to 1}.

\begin{proposition}
\label{proposition: spinal is self-similar etc}
    For every $r\geq 1$, the group $G_{H,r}$ is self-similar, strongly fractal and level-transitive.
\end{proposition}

\begin{proof}
    Since the sections of the generators are elements in $G_{H,r}$ the group $G_{H,r}$ is self-similar.
    Moreover, as $b_{i,j}\in\St_{G_{H,r}}(1)$ and $\varphi_m(B_i)=B_i$, we get $B_i\leq \varphi_m(\St_{G_{H,r}}(1))$ for $1\le i\le l$. On the other hand, conjugating $b_{i,r}$ with the appropriate element in $H$, we can obtain an element of $\St_{G_{H,r}}(1)$ with $x_i$ as a section in the vertex $m$. Thus, we also get $H\leq \varphi_m(\St_{G_{H,r}}(1))$ and we conclude that $\varphi_{m}(\St_{G_{H,r}}(1))=G_{H,r}$.  Since $H$ is transitive on $\{1,\dots, m\}$, it follows that $G_{H,r}$ is transitive on the first level, and thus $\varphi_{k}(\St_{G_{H,r}}(1))=G_{H,r}$ for $1\le k\le m$. This shows that $G_{H,r}$ is strongly fractal. Level-transitivity follows by being transitive on the first level and fractal.
\end{proof}

\begin{proposition}
\label{proposition: sn=0 for n at most r}
    The $(r+1)$st congruence quotient of $G_{H,r}$ is isomorphic to $H\wr\overset{r+1}{\dotsb}\wr H$. In particular $\mathfrak{s}_n=0$ for $n\le r$.
\end{proposition}
\begin{proof}
    It follows from the fact that the generators of $G_{H,r}$ give a generating set of $H\wr\overset{r+1}{\dotsb}\wr H$ under the epimorphism
    \begin{align*}
        G_{H,r}&\twoheadrightarrow \frac{G_{H,r}}{\mathrm{St}_{G_{H,r}}(r+1)}\overset{\cong}{\to} H\wr\overset{r+1}{\dotsb}\wr H.\qedhere
    \end{align*}
\end{proof}

Contrary to \v{S}uni\'{c} groups, the groups $G_{H,r}$ do not always contain an $m$-cycle as a rooted automorphism, which complicates finding a branch structure. Therefore, we shall distinguish two cases: when $H$ contains an element having a long enough cycle in its cycle decomposition and when it does not. We see how to find a branch structure in the first case and reduce the second case to the first one later.

Let us assume that $H$ contains an element $\sigma$ whose cycle decomposition contains a cycle of length $c$, where $ 2\ell+1\le c \le m$. Without loss of generality let us assume it is the cycle $(m\, 1\, 2\,\dotsb\, c-1)$, otherwise we would change $\{B_{i}\}_{1\le i\le \ell}$ accordingly.

\begin{proposition}
\label{proposition: sunic branch over commutator}
    Let $H\leq Sym(m)$ be such that $(m\, 1\, 2\,\dotsb\, c-1)\in H$. Then the group $G_{H,r}$ is regular branch over $G_{H,r}'$. 
\end{proposition}
\begin{proof}
    Since $G_{H,r}$ is generated by finitely many torsion elements, $G_{H,r}$ has finite abelianization. Therefore, as $G_{H,r}$ is self-similar and level-transitive by \cref{proposition: spinal is self-similar etc}, it is enough to prove $G_{H,r}'$ is a branching subgroup. Although \cite[Proposition~2.18]{GGS},  is stated for GGS-groups,   the proof works for every strongly fractal level-transitive group, and so, it is enough to prove that for any pair of generators $g_1,g_2\in G_{H,r}$ we have  $[g_1,g_2]*m\in G_{H,r}'$. For $1\le i,k\le \ell$ and $1\le j,t\le r$ it is clear that
    $$[b_{i,j},b_{k,t}]*m=[b_{i,j-1},b_{k,t-1}]\in G_{H,r}',$$
    where $j-1$ and $t-1$ are taken modulo $r$. Now, let $h\in G_{H,r}$ be such that $h$ is rooted and acts on the $h$-orbit of the vertex $m$ as the cycle $(m\,1\,2\,\dotsb\, c-1)$. Then for any $1\le i,k\le \ell$ and $1\le j\le r$ we have
    $$[b_{i,j},x_k]*m=[b_{i,j-1},b_{k,r}^{h^{-k}}]\in G_{H,r}',$$
    as $i+k\le 2\ell<c$ by assumption so there is no overlapping at the other components, where if $j=1$ we take $j-1=r$. Lastly, for $1\le i<k\le \ell$ we have
    \begin{align*}
        [x_i,x_k]*m&=[b_{i,r}^{h^{-i}},b_{k,r}^{h^{-k}}]\in G_{H,r}.\qedhere
    \end{align*}     
\end{proof}

\begin{proposition}\label{proposition: sunic H-dim}
    The Hausdorff dimension of the closure of $G_{H,r}$ in $W_H$ is at least 
    $$1-\frac{r+1}{m^{r-1}}.$$ 
    In particular, the above value tends to 1 as $r$ tends to infinity. 
\end{proposition}
\begin{proof}
    By \cref{proposition: sn=0 for n at most r} we have $\mathfrak{s}_n=0$ for $n\le r$. Furthermore, for any $n\ge r+1$ we have
    \begin{align*}
        \mathfrak{s}_{n}=\mathfrak{r}_{n}&\le \log|G_{H,r}\times\dotsb\times G_{H,r}:\psi(\mathrm{St}_{G_{H,r}}(1))|\\
        &\le m\log|G_{H,r}:G_{H,r}'|\\
        &\le m(m-1)(r+1)\log|H|,
    \end{align*}
    by \cref{proposition: sunic branch over commutator}. Therefore, as $G_{H,r}$ is self-similar, by \cref{thm: Jorge_strong} we get
    \begin{align*}
        \mathrm{hdim}_{W_H}(\overline{G}_{H,r})&=\frac{1}{\log|H|}\big(\log|G_{H,r}:\mathrm{St}_{G_{H,r}}(1)|-S_{G_{H,r}}(1/m)\big)\\
        &=1-\frac{1}{\log|H|}\sum_{n\ge 1}\frac{\mathfrak{s}_n}{m^n}\\
         &\ge 1-m(m-1)(r+1)\sum_{n\geq r+1}\frac{1}{m^n}\\
         &= 1-\frac{m(m-1)(r+1)}{m^{r+1}}\sum_{n\ge 0}\frac{1}{m^n}\\
    &= 1-\frac{r+1}{m^{r-1}}.\qedhere
    \end{align*}
\end{proof}

Given the $m$-adic tree $T_m$, for any $n\ge 1$ the $m^n$-adic tree $T_{m^n}$ may be obtained from $T_m$ by \textit{deletion of levels}, i.e. by deleting the levels which are not multiples of~$n$. This deletion of levels induces a canonical embedding $\phi:\mathrm{Aut}~T_m\hookrightarrow \mathrm{Aut}~T_{m^n}$, which allows us to see a subgroup $G\le \mathrm{Aut}~T_m$ as a subgroup of $\mathrm{Aut}~T_{m^n}$. This embedding yields a canonical isomorphism $W_H\overset{\cong}{\to} \phi(W_H)=W_{H\wr\overset{n}{\dotsb}\wr H}$ for any $n\ge 1$ and any $H\le \mathrm{Sym}(m)$. Note that the monomorphism $\phi$ commutes with taking sections at the non-deleted levels, i.e. $\phi(g)_u=\phi(g_u)$ for every $g\in \mathrm{Aut}~T_m$ and every $u\in T_{m^n}$, where we identify $u\in T_{m^n}$ with the corresponding vertex in $T_m$ before deleting the levels.

\begin{definition}
    We define the self-similar closure of a group $G\leq \Aut T_m$, and denote it by $G_\mathrm{SS}$, as the smallest self-similar subgroup of $\mathrm{Aut}~T_m$ containing $G$. Namely
    $$G_\mathrm{SS}=\langle g_u \mid g\in G, u\in\mathrm{Aut}~ T_m\rangle.$$
\end{definition}

The following lemmata will be key to reduce the general case in \cref{lemma: existence of groups tending to 1} to the one considered in \cref{proposition: sunic branch over commutator}. We consider a finitely generated self-similar group $G\le  W_{H\wr\overset{n}{\dotsb}\wr H}\le \mathrm{Aut}~T_{m^n}$ for some $n\ge 1$. Then, we show that taking the self-similar closure of $\phi^{-1}(G)$ in $W_H\le \mathrm{Aut}~T_m$ preserves finite generation and does not decrease the Hausdorff dimension of the respective topological closure. The proof relies heavily on the strongness of the Hausdorff dimension of self-similar profinite groups given by \cref{thm: Jorge_strong}.

\begin{lemma}
    \label{lemma: self-similar closure trick finite generation}
    Let $H\le \mathrm{Sym}(m)$ and let $G\le W_{H\wr\overset{n}{\dotsb}\wr H}$ $\leq \Aut T_{m^n}$ be a finitely generated self-similar group. Then the self-similar closure of $\phi^{-1}(G)$ in $W_H\leq \Aut T_{m}$ is finitely generated and 
    $$d(\phi^{-1}(G)_\mathrm{SS})\le \frac{m^{n}-1}{m-1}d(G).$$
\end{lemma}
\begin{proof}
It suffices to show that if $G=\langle S\rangle$ with $|S|=d(G)$, then $\phi^{-1}(G)_\mathrm{SS}=\widetilde{G}$, where
$$\widetilde{G}:=\langle \phi^{-1}(s)|_u\mid s\in S \text{ and }u\in T_m^{[n-1]}\rangle.$$
Clearly, the generating set of $\widetilde{G}$ is a subset of the generating set of $\phi^{-1}(G)_\mathrm{SS}$, so $\widetilde{G}\leq \phi^{-1}(G)_\mathrm{SS}.$ 
On the other hand
$$\phi^{-1}(G)=\langle \phi^{-1}(s)\mid s\in S\rangle=\langle \phi^{-1}(s)|_{\emptyset}\mid s\in S\rangle,$$
so the generating set of $\phi^{-1}(G)$ is a subset of the generating set of $\widetilde{G}$, showing that $\phi^{-1}(G)\leq \widetilde{G}$. Now, it suffices to justify that $\widetilde{G}$ is self-similar, as then by minimality of the self-similar closure we obtain that $\phi^{-1}(G)_{SS}\leq \widetilde{G}$ which shows the desired equality.
To conclude that $\widetilde{G}$ is self-similar it is enough to check that $(\phi^{-1}(s)|_u)|_x\in \widetilde{G}$ for every $x\in\{1,\dotsc,m\}$ and $u\in T_m^{[n-1]}$. The result is clearly true if $u\in T_m^{[n-2]}$. If $u\in \mathcal{L}_{n-1}$, then $s|_{ux}=g$ for some $g\in G$ by self-similarity of $G$. But now, as $G$ is generated by $S$, we get 
\begin{align*}
    (\phi^{-1}(s)|_u)|_x&=\phi^{-1}(s|_{ux})=\phi^{-1}(g)\in\langle \phi^{-1}(s)\mid s\in S\rangle =\phi^{-1}(G)\leq\widetilde{G}.\qedhere
\end{align*}
\end{proof}

\begin{lemma}
\label{lemma: self-similar closure trick H-dim}
    Let $H\le \mathrm{Sym}(m)$ be a transitive group and let $G\le W_{H\wr\overset{n}{\dotsb}\wr H}$ $\leq \Aut T_{m^n}$. Then 
    $$\mathrm{hdim}_{W_H}^{\mathcal{S}}(\overline{\phi^{-1}(G)_{\mathrm{SS}}})\ge \mathrm{hdim}_{W_{H\wr\overset{n}{\dotsb}\wr H}}^{\mathcal{S}}(\overline{G_{\mathrm{SS}}}).$$
\end{lemma}
\begin{proof}
Let us assume without loss of generality that $G$ is self-similar to simplify notation. Then $\overline{G}$ has strong Hausdorff dimension in $W_{H\wr\overset{n}{\dotsb}\wr H}$ by \cref{thm: Jorge_strong} and we get 
$$\mathrm{hdim}_{W_{H\wr\overset{n}{\dotsb}\wr H}}^\mathcal{S}(\overline{G})=\lim_{k\to \infty}\frac{\log|G:\mathrm{St}_{G}(k)|}{\log|W_{H\wr\overset{n}{\dotsb}\wr H}:\mathrm{St}_{W_{H\wr\overset{n}{\dotsb}\wr H}}(k)|}.$$
Note also that
$$\frac{\log|\phi^{-1}(G)_{\mathrm{SS}}:\mathrm{St}_{\phi^{-1}(G)_{\mathrm{SS}}}(nk)|}{\log|W_H:\mathrm{St}_{W_H}(nk)|}\ge \frac{\log|G:\mathrm{St}_{G}(k)|}{\log|W_{H\wr\overset{n}{\dotsb}\wr H}:\mathrm{St}_{W_{H\wr\overset{n}{\dotsb}\wr H}}(k)|}$$
for every $k\ge 1$. Furthermore $\overline{\phi^{-1}(G)_{\mathrm{SS}}}$ has strong Hausdorff dimension in $W_H$ again by self-similarity and \cref{thm: Jorge_strong}. Therefore the limit of any subsequence of the above logarithmic quotients coincides with the limit of the full sequence of logarithmic quotients, so we get
\begin{align*}
    \mathrm{hdim}_{W_H}^\mathcal{S}(\overline{\phi^{-1}(G)_{\mathrm{SS}}})&= \lim_{k\to\infty}\frac{\log|\phi^{-1}(G)_{\mathrm{SS}}:\mathrm{St}_{\phi^{-1}(G)_{\mathrm{SS}}}(nk)|}{\log|W_H:\mathrm{St}_{W_H}(nk)|}\\
    &\ge \lim_{k\to\infty}\frac{\log|G:\mathrm{St}_{G}(k)|}{\log|W_{H\wr\overset{n}{\dotsb}\wr H}:\mathrm{St}_{W_{H\wr\overset{n}{\dotsb}\wr H}}(k)|}\\
&=\mathrm{hdim}_{W_{H\wr\overset{n}{\dotsb}\wr H}}^\mathcal{S}(\overline{G}).\qedhere
\end{align*}
\end{proof}

\begin{proof}[Proof of \cref{lemma: existence of groups tending to 1}]
If $H$ contains an element $\sigma$ having a cycle of length $c\ge 2\ell +1$ the generalized \v{S}uni\'{c} groups $\{G_{H,r}\}_{r\ge 1}$ are self-similar and finitely generated satisfying both that $\lim_{r\to\infty }\mathrm{hdim}_{W_H}^\mathcal{S}(\overline{G}_r)=1$ and that $d(G_r)$ is linear on $r$. Therefore let us assume $H$ does not contain such an element. However, since $H$ is non-trivial it contains an element having a cycle of length at least 2 in its cycle decomposition. Now, the wreath product $H\wr\overset{n}{\dotsb}\wr H$ is transitive on the $n$th level of the tree and $d(H\wr\overset{n}{\dotsb}\wr H)\le nd(H)\le n(m-1)$ elements, i.e. its rank is linear on $n\ge 1$. However, the group $H\wr\overset{n}{\dotsb}\wr H$ contains an element having a cycle of length at least $2^n$ in its cycle decomposition. Therefore, for big enough $n$ we get that $H\wr\overset{n}{\dotsb}\wr H$ contains an element having a cycle of length at least $2\cdot d(H\wr\overset{n}{\dotsb}\wr H)+1$ in its cycle decomposition. Let us consider $n$ minimal satisfying this. We note that this minimal $n$ is a constant depending only on $H$. Now, by \cref{proposition: sunic H-dim} the family of self-similar subgroups $G_{H\wr\overset{n}{\dotsb}\wr H,r}\le W_{H\wr\overset{n}{\dotsb}\wr H}$ satisfies for every $r\ge 1$ that
$$\mathrm{hdim}_{W_{H\wr\overset{n}{\dotsb}\wr H}}^\mathcal{S}(\overline{G_{H\wr\overset{n}{\dotsb}\wr H,r}})\ge 1-\frac{r+1}{m^{n(r-1)}}.$$

Then by \cref{lemma: self-similar closure trick finite generation} and \cref{lemma: self-similar closure trick H-dim} we get a family of self-similar groups $\phi^{-1}(G_{H\wr\overset{n}{\dotsb}\wr H,r})_{\mathrm{SS}}\le W_H$ such that for every $r\ge 1$
\begin{align*}
    \mathrm{hdim}_{W_H}^\mathcal{S}(\overline{\phi^{-1}(G_{H\wr\overset{n}{\dotsb}\wr H,r})_{\mathrm{SS}}})&\ge \mathrm{hdim}_{W_{H\wr\overset{n}{\dotsb}\wr H}}^\mathcal{S}(\overline{G_{H\wr\overset{n}{\dotsb}\wr H,r}})\ge 1-\frac{r+1}{m^{nr-n}}
\end{align*}
and $d(\phi^{-1}(G_{H\wr\overset{n}{\dotsb}\wr H,r})_{\mathrm{SS}})\le \big((m^{n+1}-1)/(m-1)\big) d(G_{H\wr\overset{n}{\dotsb}\wr H,r})=Cr$ for some constant $C$ depending only on $H$.
\end{proof}

\subsection{Siegenthaler's construction}

Now we extend Siegenthaler's construction of a 1-dimensional finitely generated subgroup of $W_{2}$ in \cite{SiegenthalerBinary} to any iterated wreath product $W_H$ where $H\le \mathrm{Sym}(m)$ is any transitive subgroup. We follow the same approach as in \cite[Section 6]{SiegenthalerBinary}.

We first fix some notation. Let $H=\langle h_1,\dotsc, h_\ell\rangle$. We define recursively
$$\psi(a_j)=(a_j,1,\dotsc,1)h_{j}$$
for $1\le j\le \ell$. We shall consider $a_{n,j}$ the truncated version of $a_{j}$ at every level $n\ge 1$, i.e. 
$$\psi(a_{1,j})=(1,\dotsc,1)h_j$$
and
$$\psi(a_{n,j})=(a_{n-1,j},1,\dotsc,1)h_j$$
for $n\ge 2$. Let $r$ be such that $h_r$ moves the vertex $1$. A straightforward computation shows that the $a_{n,r}$-orbit of $1\overset{n}{\dotsb}1$ at the $n$-th level is of length $t_r^n$, where $t_r$ is the length of the $h_r$-orbit of $1$. Then for every $g_1,\dotsc,g_n\in W_H$ we define
$$\delta(g_1,\dotsc,g_n):=\prod_{i=0}^{n-1}g_{i+1}*\big((1\overset{n}{\dotsb}1)a_{n,1}^{{t_r^{i}}}\big).$$

\begin{lemma}
\label{lemma: Siegenthaler H' in rist}
    Let $G=\langle g_1,\dotsc,g_n\rangle\le W_H$ be any finitely generated subgroup and let us consider $K=\langle a_{n,j},\delta(g_1,\dotsc,g_n)\mid 1\le j\le \ell\rangle\le W_H$. Then $$G'*\mathcal{L}_n\le K'.$$
\end{lemma}
\begin{proof}  
    We first note that
    $$\delta(g_1,\dotsc,g_n)^{a_{n,r}^{-t_r^s}}=\prod_{i=0}^{n-1}g_{i+1}*\big((1\overset{n}{\dotsb}1)a_{n,1}^{t_r^{i}-t_r^s}\big)$$
    for every $0\le s\le n-1$. Thus, arguing as in the proof of \cref{lemma: construction of LI}, we get
    $$[\delta(g_1,\dotsc,g_n)^{a_{n,r}^{-t_r^s}},\delta(g_1,\dotsc,g_n)^{a_{n,r}^{-t_r^u}}]=[g_{s+1},g_{u+1}]*(1\overset{n}{\dotsb}1)$$
    for any $0\le s,u\le n-1$. Now, for any $g\in G$ we may conjugate the above commutator with an appropriate element in $K$ as in the proof of \cref{lemma: construction of LI} to obtain
    $$[g_{s+1},g_{u+1}]^g*(1\overset{n}{\dotsb}1)\in K',$$
    i.e. we obtain
    $$G'*(1\overset{n}{\dotsb}1)\le K'$$
    as $G'=\langle [g_s,g_u]\mid 1\le s,u\le n\rangle^G$. Finally, since $H$ acts transitively on the first level the subgroup $\langle a_{n,j}\mid 1\le j\le \ell\rangle\le K$ acts transitively on the $n$th level and we obtain that
    \begin{align*}
        G'*\mathcal{L}_n&\le K'.\qedhere
    \end{align*}
\end{proof}

\begin{proof}[Proof of \cref{lemma: Siegenthaler generalization}]
    For each $r\ge 1$ we consider the family of groups $\{G_r\}_{r\ge 1}$ obtained in \cref{lemma: existence of groups tending to 1} and let us fix a set of generators  $G_{r}=\langle g_1,\dotsc, g_{d(G_r)}\rangle$. For each $r\ge 1$, we define the subgroup 
    $$K_r:=\langle a_{d(G_r)+\ell+1,j}, b_{r}\mid 1\le j\le \ell\rangle,$$
    where 
    $$b_r:=\delta(g_1,\dotsc, g_{d(G_r)}, a_{d(G_{r+1})+\ell+1,1},\dotsc,a_{d(G_{r+1})+\ell+1,\ell}, b_{r+1}).$$ 
    Then, by \cref{lemma: Siegenthaler H' in rist} we have
    \begin{align*}
        K_r'\ge K_{r+1}'*\mathcal{L}_{d(G_r)+\ell+1}\quad\text{and}\quad K_r'\ge G'_{H,r}*\mathcal{L}_{d(G_r)+\ell+1}
    \end{align*}
    for every $r\ge 1$. In particular
    \begin{align}
    \label{align: inequalities of siegenthaler}
        K_1\ge K_1'\ge G'_{H,r}*\mathcal{L}_{\sum_{i=1}^r(d(G_i)+\ell+1)}
    \end{align}
    for every $r\ge 1$. Finally since $G_{H,r}$ is self-similar for every $r\ge 1$, its closure in $W_H$ has strong Hausdorff dimension by \cref{thm: Jorge_strong}. Thus by \textcolor{teal}{(}\ref{align: inequalities of siegenthaler}\textcolor{teal}{)} and \cref{proposition: dimension is a series} we get
    \begin{align*}
        1\ge \mathrm{hdim}_{W_{H}}^\mathcal{S}(\overline{K}_1)&\ge \sup_{r\ge 2}\mathrm{hdim}_{W_{H}}^\mathcal{S}(\overline{G'_{H,r}})=\lim_{r\to\infty}\mathrm{hdim}_{W_{H}}^\mathcal{S}(\overline{G_{H,r}})=1,
    \end{align*}
    where the second inequality follows by the monotonicity of the Hausdorff dimension, the first equality by \cite[Proposition 3.1]{JorgeMikel}, and the second equality by \cref{lemma: existence of groups tending to 1}.
\end{proof}

\subsection{2-generated subgroups of $W_p$ with prescribed Hausdorff dimension}

To conclude the section, we prove  \cref{Corollary: Wp 2-bounded spectrum}, giving further evidence to a conjecture of Abért and Virág; see \cite[Conjecture 7.3]{AbertVirag}.

We set $\sigma:=(1\,\dotsb \, p)\in \mathrm{Sym}(p)$ and $\psi(a):=(a,1,\dotsc,1)\sigma$. The automorphism~$a$ is the so-called \textit{adding machine} and we denote by $a_n$ the $n$th truncated version of~$a$ for every level $n\ge 1$. It is immediate to see that $a$ acts transitively on the $p$-adic tree and $a_n$ acts transitively on the first $n$ levels of the $p$-adic tree.

Arguing as in \cref{lemma: construction of LI}, we fix a set of non-comparable vertices $X=\{x_n\}_{n\ge 1}$ and we define both the automorphisms
$$g:=\prod_{n\ge 1}a_n*x_n\quad\text{and}\quad h:=\Big(\prod_{n\ge 1}k_1*x_n(1\overset{n}{\dotsb}1)\Big)\Big(\prod_{n\ge 1}k_2*x_n (1\overset{n}{\dotsb}1)a_n^{p^{n-1}}\Big),$$
where $K:=\langle k_1,k_2\rangle$ is the group constructed in \cref{theorem: existence of 1-dimensional subgroups in WH} for the $p>2$ case, while for the $p=2$, it is the one constructed by Siegenthaler in \cite{SiegenthalerBinary}. Note that the closure of $K$ is $2$-generated and has Hausdorff dimension 1 in $W_H$. Moreover, it satisfies that

$$\prod_{n\ge 1}\overline{H'}*x_n{Y_n}\le \overline{L}_X\le W_p*X,$$
for some transversal $Y_n$ of $T_p$, by the same reasoning as in the proof of \cref{lemma: construction of LI}. Therefore, by \cite[Theorem 5]{AbertVirag} and \cref{proposition: dimension is a series} we get
$$\mathrm{hdim}_{W_p}^\mathcal{S}(\overline{L}_X)=\mu(X)$$
where $\mu(X)$ can be chosen to be any number in $[0,1]$. This proves \cref{Corollary: Wp 2-bounded spectrum}.

\section{Standard filtrations: the first Grigorchuk group}
\label{subsection: grigorchuk group}
As it was mentioned in the introduction, the level-stabilizers filtration $\mathcal{S}$ is the most natural choice for the definition of the Hausdorff dimension of profinite groups acting on rooted trees. However, in the case of pro-$p$ groups one is also interested in some standard filtrations. Indeed, Shalev's problem was originally posed for the $p$-power filtration and later considered for other standard filtrations; compare \cite{IkerAnitha, ShalevNewHorizons}. Here, we shall consider the closure $\Gamma$ of the first Grigorchuk group $\mathfrak{G}$ and two of the most natural standard filtrations on pro-$p$ groups: the $p$-central lower series $\mathcal{L}$ and the Jenning-Zassenhaus series $\mathcal{D}$.

Let $\mathfrak{G}$ denote the first Grigorchuk group which acts on the binary rooted tree $T_2$ and is generated by four elements; say, $a,b,c,d$ where $a$ is the rooted automorphism $a=(1,1)(1\,2)\in (\Aut T_2\times\Aut T_2)\rtimes\mathrm{Sym}(2)$, i.e. $a$ permutes the two subtrees hanging below the root, and the rest of the elements are defined recursively as 
\[
\psi(b):=(a,c), \quad \psi(c):=(a,d)\quad \text{and}\quad  \psi(d):=(1,b).
\]
By \cite[Proposition 4.2]{BartholdiGrigorchuk2} the group $\mathfrak{G}$ is  regular branch over the normal subgroup 
$$K:=\langle x \rangle^G=\langle x, (x,1), (1,x)\rangle,$$
where $x:=[a,b]=ababa$.
Consider  the lower 2-central series $\mathcal{L}$ (which coincides with the lower central series of $\mathfrak{G}$) and the Jenning-Zassenhaus series $\mathcal{D}$ of $\mathfrak{G}$ defined as follows:
$$\mathcal{L}: \gamma_1(\mathfrak{G}):=\mathfrak{G}, \;\gamma_{n+1}(\mathfrak{G}):=[\gamma_{n}(\mathfrak{G}),\mathfrak{G}]\gamma_{n}(\mathfrak{G})^2$$

$$\mathcal{D}: D_1(\mathfrak{G}):=\mathfrak{G},\; D_n(\mathfrak{G}):=[D_{n-1}(\mathfrak{G}),\mathfrak{G}]D_{\lceil n/2\rceil}(\mathfrak{G})^2=\prod_{i2^j\geq n}\gamma_i(\mathfrak{G})^{2^j}.$$

We are interested in computing the Hausdorff dimension of the closure $\Gamma$ of $\mathfrak{G}$ in $W_2$ with respect to $\mathcal{S}$. However, the following lemma and corollary show we may work with $\mathfrak{G}$: 

\begin{lemma}
\label{lemma: verbal subgroups}
    Let $G\le \mathrm{Aut}~T_\mathbf{m}$ be a finitely generated group and $w(G)\le G$ an open verbal subgroup. Then $\overline{w(\overline{G})}=\overline{w(G)}$ and as a consequence, $\overline{G}/w(\overline{G})\cong G/w(G)$.
\end{lemma}
\begin{proof}
    The containment $\overline{w(\overline{G})}\ge \overline{w(G)}$ is immediate. The converse follows because if $w(G)$ is open, there exists $k\ge 1$ such that $w(G)\ge \mathrm{St}_G(k)$. But $\overline{G}/\mathrm{St}_{\overline{G}}(k)\cong G/\mathrm{St}_G(k)$ and since verbal subgroups are preserved under taking quotients we obtain both $\overline{w(G)}$ and $\overline{w(\overline{G})}$ have the same index in $\overline{G}$, concluding the proof.
\end{proof}

\begin{corollary}
    Let $G\le \mathrm{Aut}~T_\mathbf{m}$ be a finitely generated branch group with the congruence subgroup property. Then, for every non-trivial word $w$, the following equality holds:
    \[
    |G:w(G)|=|\overline{G}:w(\overline{G})|.
    \]
\end{corollary}
\begin{proof}
    By Theorem 4.4.4 in \cite{Dominik:PhD_thesis} the group $G$ is just-infinite. Since $w(G)$ is normal in $G$ and by  \cite[Corollary 1.4]{Abert} we have $w(G)\ne 1$, the subgroup $w(G)$ is of finite index and thus open by the congruence subgroup property. Then the result follows by \cref{lemma: verbal subgroups}.
\end{proof}

From now on we shall only work on the binary rooted tree. Thus, in the remainder of the section, the integer $m\ge 1$ will not longer represent the degree of a rooted tree but it will be used as an index for certain filtrations of $\mathfrak{G}$.

In \cite{Rozhkov}, Rozhkov computed the successive quotients of $\mathcal{L}$ for $\mathfrak{G}$. To that aim, he defined a filtration $\{N_m\}_{m\ge 0}$ of $\mathfrak{G}$ which we introduce below. We shall use the following notation: for any $Q\le K$ and $n\ge 0$ we write
\[Q_n:=\psi_n^{-1}(Q\times \overset{2^n}{\cdots}\times Q)\le \St_G(n).\]
Using this notation, 
set $T:=\langle x^2\rangle^\mathfrak{G}=K^2$ and define
$$\mathcal{N}: N_0:=\mathfrak{G} \text{ and } \text{ for each }\; m\geq 1, \quad N_m=K_m\cdot T_{m-1}.$$ 
By \cite[Lemma 6.1]{BarthGrig}, we have 
$$\frac{N_m}{N_{m+1}}=X\times Y\cong \mathbb{F}_2^{2^m}\times \mathbb{F}_2^{2^{m-1}},$$
where 
$$X:=\langle (1,\overset{r}{\dotsc},1,x,1,\dotsc,1)N_{m+1}\mid 0\le r\le 2^m-1\rangle \le \mathbb{F}_2^{2^{m}}$$
and
$$Y:=\langle (1,\overset{r}{\dotsc},1,x^2,1,\dotsc,1)N_{m+1}\mid 0\le r\le 2^{m-1}-1\rangle \le \mathbb{F}_2^{2^{m-1}}.$$

Rozhkov's methods were generalised in \cite{BarthGrig} to a wider class of regular branch groups and in what follows we recall some of their results.

Define, for every integer $m\ge 1$, the $\mathbb{F}_2$-vector space
$$V_m:=\frac{\mathbb{F}_2[X_1,\dotsc, X_m]}{(X_1^2-1,\dotsc, X_m^2-1)}\cong\frac{\mathbb{F}_2[X_1]}{(X_1^2-1)}\otimes\dotsb \otimes \frac{\mathbb{F}_2[X_m]}{(X_m^2-1)}.$$
We also set $V_0:={0}$. For $0\le r\le 2^m-1$, we write $r=(r_m\dotsb r_1)_2$ in base 2, i.e. $r=\sum_{i=1}^m r_i 2^{i-1}$, and define the elements
$$v_m^r:=(1-X_1)^{r_1}\dotsb (1-X_m)^{r_m}.$$
Furthermore, we define the following $\F_2$-vector subspaces of $V_m$:
$$V_m^r:=\langle v_m^i\mid r\le i\le  2^m-1\rangle$$
and set $V_m^r={0}$ for every $r\ge 2^m$. Note that the inclusions
\[
V_m^{2^m}=0\subset \dots \subset V_m^1\subset V_m^0=V_m
\]
hold, and that $V_m^{j+1}$ has codimension one in $V_m^{j}$ for all $0\le j\le 2^m-1$. 

We fix an integer $m\geq 1$ and we write $\{X_1, \dots, X_m\}$
and $\{Y_1, \dots, Y_{m-1}\}$ for the monomials in the direct sum $V_m\oplus V_{m-1}$, respectively. We define the $\mathbb{F}_2$-linear maps
    \[\begin{matrix}
    \alpha \colon & X\le N_m/N_{m+1}  & \longrightarrow &V_m\\ &&&\\
&(1,\overset{r}{\dotsc},1,x,1,\dotsc,1)N_{m+1}&\mapsto &X_1^{r_1}\dotsb X_m^{r_m},
    \end{matrix}\]
    and 
    \[\begin{matrix}
    \beta \colon & Y\le N_m/N_{m+1}  & \longrightarrow &V_{m-1}\\ &&&\\
&(1,\overset{r}{\dotsc},1,x^2,1,\dotsc,1)N_{m+1}&\mapsto &Y_1^{r_1}\dotsb Y_{m-1}^{r_{m-1}}.
    \end{matrix}\]

\begin{lemma}[{see {\cite[Lemma 6.1]{BarthGrig}}}]
    The $\mathbb{F}_2$-linear map \[
    \alpha\oplus \beta\colon N_m/N_{m+1}=X\times Y\to V_m\oplus V_{m-1}\] is an $\mathbb{F}_2$-isomorphism.
  
\end{lemma}

An immediate consequence of the above result is that for $\mathrm{rist}_{N_m}(i_1\overset{k}{\ldots}i_k)$ we can restrict the map $\alpha\oplus \beta$ to an isomorphism by fixing the prefix monomials $X_{1}^{i_{1}}\dotsb X_k^{i_k}$ and $Y_{1}^{i_{1}}\dotsb Y_k^{i_k}$: 

\begin{corollary}
\label{corollary: rist in terms of Vn}
    There is an $\mathbb{F}_2$-isomorphism
        \begin{align*}
        \alpha\oplus \beta\colon\mathrm{rist}_{N_m}(i_1\overset{k}{\ldots}i_k)N_{m+1}/N_{m+1} &\to V_{m-k}\oplus V_{m-1-k}\\
        (1,\overset{r}{\dotsc},1,x,1,\dotsc,1)N_{m+1}&\mapsto X_{k+1}^{r_{k+1}}\dotsb X_m^{r_m},\\
        (1,\overset{r}{\dotsc},1,x^2,1,\dotsc,1)N_{m+1}&\mapsto Y_{k+1}^{r_{k+1}}\dotsb Y_{m-1}^{r_{m-1}},
        \end{align*}
where we identify $X_{l}$ with $X_{l-k}$ and $Y_l$ with $Y_{l-k}$ for $1\le l\le m$.

\end{corollary}

\begin{proposition}[{see {\cite[Theorems 6.4 and 6.6]{BarthGrig}}}]
\label{proposition: lower central and dimension series of grigorchuk group}
    For all $m\ge 1$ we have:
    \begin{enumerate}[\normalfont(i)]
        \item the equality of subgroups $\gamma_{2^m+1}(\mathfrak{G})=D_{2^m+1}(\mathfrak{G})=N_m$;
        \item the equality of subgroups $\gamma_{2^m+1+r}(\mathfrak{G})=\alpha^{-1}(V_m^r)\beta^{-1}(V_{m-1}^r)N_{m+1}$ for any $1\le r\le 2^m-1$;
        \item the equality of subgroups $$D_{2^m+1+r}(\mathfrak{G})=\begin{cases}
            \alpha^{-1}(V_m^r)\beta^{-1}(V_{m-1}^{r/2})N_{m+1}$ for $2\le r\le 2^m-1 \text{ even},\\
            \alpha^{-1}(V_m^r)\beta^{-1}(V_{m-1}^{(r-1)/2})N_{m+1}$ for $1\le r\le 2^m-1 \text{ odd.}
        \end{cases}$$
    \end{enumerate}
\end{proposition}

We prove \cref{theorem: Grigorchuk spectra} for $\mathcal{L}$ as, thanks to \cref{proposition: lower central and dimension series of grigorchuk group}, the proof for $\mathcal{D}$ is completely analogous. It is enough to prove the equivalent of \cref{lemma: dimension of product} for $\mathcal{L}$; see \cref{proposition: grigorchuk rist dimension} below. Then we may apply the machinery in \cref{proposition: dimension is a series} and the construction in \cref{lemma: construction of LI} to prove \cref{theorem: Grigorchuk spectra}, as the commutator subgroup of any normal subgroup of $\mathfrak{G}$ is still of finite index in $\mathfrak{G}$ by just-infiniteness of $\mathfrak{G}$ and the Hausdorff dimension does not change when passing to a finite index subgroup.

We introduce the following notation. For a vertex $v\in T_2$ we consider the finite path $x_1\dotsb x_n$ from the root to $v$ and write $(v)_2$ for the natural number whose binary expansion is precisely $(x_n\dotsb x_1)_2$ (for this purpose we label the tree vertices by the alphabet $\{0,1\}$ instead of $\{1,2\}$). In other words, 
$$(v)_2:=(x_n\dotsb x_1)_2=x_n\cdot 2^{n-1}+\dotsb +x_2\cdot 2+x_1.$$
Let also $|v|_w$ denote the word metric on $T_2$, i.e. $|v|_w=n$ if $v$ lies at the $n$th level of $T_2$.

\begin{lemma}\label{lemma:Nm_product_filtration}
    The filtration series $\{N_m\}_{m\geq 0}$ is a direct product filtration.
\end{lemma}
\begin{proof}
    It suffices to prove that $N_m=\psi^{-1}(N_{m-1}\times N_{m-1})$ for $m\geq 2$. We have 
\begin{align*}
    \psi^{-1}(N_{m-1}\times N_{m-1})&=\psi^{-1}(K_{m-1}T_{m-2}\times K_{m-1} T_{m-2})\\
    &=\psi^{-1}(K_{m-1}\times K_{m-1})\psi^{-1}(T_{m-2}\times T_{m-2})=K_mT_{m-1}=N_m
\end{align*}
\end{proof}

\begin{proposition}
\label{proposition: grigorchuk rist dimension}
        Let $\Gamma\le \mathrm{Aut}~T_2$ be the closure of the first Grigorchuk group $\mathfrak{G}$ and let $V$ be a finite set of non-comparable vertices of the 2-adic tree. Then the product $\prod_{v\in V}\mathrm{rist}_\Gamma(v)$ in $\Gamma$ has strong Hausdorff dimension given by
    $$\mathrm{hdim}^{\mathcal{L}}_{\Gamma}(\prod_{v\in V}\mathrm{rist}_\Gamma(v))=\mu(V).$$
\end{proposition}
\begin{proof}
    First note that $\mathrm{rist}_\Gamma(v)=\overline{\mathrm{rist}_\mathfrak{G}(v)}$ for any vertex $v$, so we shall work with the discrete group $\mathfrak{G}$. By \cref{corollary: rist in terms of Vn}, we have the following equality:  $$\frac{\left(\prod_{v\in V} \mathrm{rist}_{N_m}(v)\right)N_{m+1}}{N_{m+1}}=\prod_{v\in V}\frac{\mathrm{rist}_{N_m}(v)N_{m+1}}{N_{m+1}},$$
    which can also be deduced from the fact that $\{N_m\}_{m\ge 0}$ is a nice filtration; see \cref{lemma:Nm_product_filtration}. Therefore, it suffices to prove the result for the $V=\{v\}$ case.

  Suppose that $d(v)=k$. By \cref{corollary: rist in terms of Vn}, we have 
    $$V_m^r\cap \alpha(\mathrm{rist}_{N_m}(v)N_{m+1}/N_{m+1})=\langle v_m^\ell\mid \ell\ge r\text{ and }\ell\equiv (v)_2~\mathrm{mod}~2^k\rangle$$
and similarly,
    $$V_{m-1}^r\cap \beta(\mathrm{rist}_{N_m}(v)N_{m+1}/N_{m+1})=\langle v_{m-1}^\ell\mid \ell\ge r\text{ and }\ell\equiv (v)_2~\mathrm{mod}~2^k\rangle.$$
Therefore, by \cref{proposition: lower central and dimension series of grigorchuk group}, we obtain that
    \begin{align*}
        \log|\mathrm{rist}_{N_m}(v):\mathrm{rist}_{\gamma_{2^m+1+r}(\mathfrak{G})}(v)|= f(r),
    \end{align*}
    where $f(r)$ is the function
    \begin{align*}
        f(r):=\begin{cases}
            2\lfloor r\cdot2^{-k}\rfloor,&\text{ if }r\le 2^{m-1}-1\\
            2^{m-1-k}+\lfloor r\cdot2^{-k}\rfloor, &\text{ if } 2^{m-1} \le r\le 2^{m}-1.
        \end{cases}
    \end{align*}
By \cref{proposition: lower central and dimension series of grigorchuk group}, note also that 
    \begin{align*}
        \log|N_m:\gamma_{2^m+1+r}(\mathfrak{G})|=g(r):=\begin{cases}
            2r,&\text{ if }r\le 2^{m-1}-1,\\
            2^{m-1}+r,&\text{ if } 2^{m-1} \le r\le 2^{m}-1.
        \end{cases}
    \end{align*}
Since $\{N_m\}_{m\ge 0}$ is a product filtration, it is a nice filtration and thus, thanks to  \cref{lemma: dimension of product}, \cref{theorem: full spectra} and \cref{corollary: full and normal spectra}, we have
   $$\lim_{m\to \infty}\frac{\log|\mathrm{rist}_{\mathfrak{G}}(v)N_m:N_m|}{\log|\mathfrak{G}:N_m|}=:\lim_{m\to \infty} \frac{x_m}{y_m}=\mu(\{v\}).$$

   Now, for $1\le r\le 2^{m-1}-1$, we have
   $$0\le \frac{f(r)}{g(r)}=\frac{\lfloor r\cdot2^{-k}\rfloor}{r}\le 2^{-k}\le 2^{1-k},$$
   and for $2^{m-1}\le r\le 2^m-1$, we have
   $$0\le \frac{f(r)}{g(r)}=\frac{2^{m-1-k}+\lfloor r\cdot2^{-k}\rfloor}{2^{m-1}+r}\le 2\cdot 2^{-k}\le 2^{1-k}.$$
   Therefore, for any $1\le r\le 2^{m}-1$ we obtain
   $$\frac{x_m\cdot g(r)^{-1}}{y_m\cdot g(r)^{-1}+1}\le \frac{x_m\cdot g(r)^{-1}+f(r)\cdot g(r)^{-1}}{y_m\cdot g(r)^{-1}+1}=\frac{x_m+f(r)}{y_m+g(r)}\le \frac{x_m\cdot g^{-1}(r)+2^{1-k}}{y_m\cdot g(r)^{-1}}$$
 and thus,
\begin{align*}
    \mathrm{hdim}_{\Gamma}^\mathcal{L}(\mathrm{rist}_\mathfrak{G}(v))=\lim_{n\to \infty}\frac{\log|\mathrm{rist}_\mathfrak{G}(v)\gamma_n(\mathfrak{G}):\gamma_n(\mathfrak{G})|}{\log|\mathfrak{G}:\gamma_n(\mathfrak{G}):\gamma_n(\mathfrak{G})|}&=\lim_{m\to \infty}\frac{x_m}{y_m}=\mu(\{v\}).\qedhere
\end{align*}
\end{proof}

\end{document}